\documentclass[12pt]{article}
\usepackage{latexsym,amsfonts,amsthm,amsmath,amscd,amssymb}
\usepackage[dvips]{graphicx}

\newtheorem{lemma}{Lemma}[section]
\newtheorem{thm}[lemma]{Theorem}
\newtheorem{rem}[lemma]{Remark}
\newtheorem{prop}[lemma]{Proposition}

\newcommand{\finedimo}{{\hfill\hbox{$\square$}\vspace{2pt}}}

\newcommand{\diff}{\,\mathrm{d}}

\newcommand\matH{{\mathbb{H}}}

\newcommand\matP{{\mathbb{P}}}
\newcommand\matR{{\mathbb{R}}}
\newcommand\matQ{{\mathbb{Q}}}
\newcommand\matZ{{\mathbb{Z}}}
\newcommand\matC{{\mathbb{C}}}

\newcommand{\vettt}[3]{\left(\begin{array}{c}#1\\#2\\#3\end{array}\right)}

\renewcommand{\hbar}{{\overline{h}}}

\newfont{\Got}{eufm10 scaled 1200}

\newcommand{\mycap} [1] {\caption{\footnotesize{#1}}}
\newcommand{\cut}{\textrm{Cut}}

\newcommand{\calT}{{\mathcal{T}}}
\newcommand{\calA}{{\mathcal{A}}}

\newcommand{\calH}{{\mathcal{H}}}
\newcommand{\calL}{{\mathcal{L}}}
\newcommand{\calP}{{\mathcal{P}}}
\newcommand{\calS}{{\mathcal{S}}}

\newcommand{\scalgen}[2]{\left\langle#1|#2\right\rangle_{(1,n)}}
\newcommand{\scal}[2]{\left\langle#1|#2\right\rangle_{(1,3)}}

\begin{document}

\title{Algorithmic construction and recognition of hyperbolic 3-manifolds, links, and graphs}

\author{Carlo~\textsc{Petronio}}

\maketitle

This survey article describes the algorithmic approaches
successfully used over the time to construct hyperbolic
structures on 3-dimensional topological ``objects'' of various types,
and to classify several classes of such objects using such structures.
Essentially, it reproduces the contents of a course given by the author at the
``Master Class on Geometry'' held in Strasbourg from
April 27 to May 2, 2009. The author warmly thanks the organizers
Norbert A'Campo, Frank Herrlich and (particularly)
Athanase Papadopoulos for having set up this excellent activity and
for having invited him to contribute to it.

\section{3-dimensional ``objects''}\label{objects:sec}

The main objects of interest in $3$-dimensional topology are
$3$-manifolds, namely topological spaces obtained by patching
together portions of Euclidean 3-space. Depending on whether
the patching is performed along continuous, differentiable, or
piecewise-linear maps, one gets the three different categories
of manifolds named TOP, DIFF, and PL, respectively. In higher dimension
these categories can differ from each other
in an essential way (for instance, one TOP manifold can
have non-diffeomorphic DIFF structures), but in dimension $3$ it has
been known for a long time
(see for instance the foundational work of Kirby and Siebenmann~\cite{kirsib}) that the three categories are equivalent
to each other. For this reason in the sequel we will use the DIFF and the PL
approaches interchangeably, the former being more suited to the discussion
of geometric structures, the latter to a combinatorial treatment.
In addition we will always view manifolds up to the natural equivalence relation in
the category in use, namely we will view two diffeomorphic or PL-equivalent manifolds
as being just one and the same object. We address the reader to
the by now classical introductions to the topic of 3-manifolds due to Hempel and to Jaco ~\cite{Hempel,Jaco}.

The most general setting of an algorithmic classification of manifolds
(or of other topological objects, as discussed below) consists of the following
ingredients:
\begin{itemize}
\item A combinatorial \emph{presentation} of the objects under consideration,
namely a way to associate a topological object to some finite set of data,
so that, given a bound on the ``complexity,'' all the relevant sets of data
can be recursively enumerated by a computer;
\item A set of \emph{moves} on the combinatorial data, by repeated applications
of which one is sure to relate to each other any two sets of data representing
the same topological object;
\item Certain \emph{invariants} of the topological objects, using which
one can (sometimes) prove one is different from another one, and perhaps also show
that they are the same (when the invariant is a complete one).
\end{itemize}

In the rest of this section we will describe some combinatorial presentations
of $3$-manifolds and of other related $3$-dimensional topological objects introduced below,
together with the corresponding moves. In the next section we will illustrate the
powerful invariants coming from the machinery of hyperbolic geometry, and in the subsequent sections
we will discuss how the combinatorial approach and the use of
the hyperbolic invariants can be (and has been) used to produce extremely
satisfactory classification results.

\paragraph{(Loose) triangulations of manifolds, and spines}
In the sequel all our manifolds will be 3-dimensional, connected, orientable, and compact
(with or without boundary). Starting from the case of a closed manifold $M$,
namely one with empty boundary, we will call (loose) \emph{triangulation}
of $M$ a realization of $M$ as the quotient of a disjoint union of standard tetrahedra
under the action of a simplicial orientation-reversing pairing of the (codimension-1) faces.
Note that a triangulation is not strictly a PL structure on $M$ according to
the original definition~\cite{RS}, because in $M$ the tetrahedra can be self-incident and
multiply incident to each other. However a loose triangulation in our sense
can be transformed into a PL structure by subdivision.
The next result (due to Matveev and to Piergallini,
see~\cite{mafo,matbook,pierg} and the references quoted therein)
describes the combinatorial approach to closed $3$-manifolds using triangulations:

\begin{thm}\label{closed:calculus}
Let $M$ be a closed orientable $3$-manifold. Then:
\begin{itemize}
\item Given $v\geq 1$ one can find triangulations of $M$ with $v$ vertices;
\item Given $v\geq 1$ and two triangulations of $M$ with $v$ vertices,
both consisting of at least two tetrahedra, one can transform them into each other
by repeated applications of the $2$-to-$3$ move shown in Fig.~\ref{triamoves:fig}-top, and its inverse;
\item One can transform any two triangulations of $M$ into each other
by repeated applications of the $2$-to-$3$ and the $1$-to-$4$ moves
shown in Fig.~\ref{triamoves:fig}, and their inverses.
\end{itemize}
\end{thm}

    \begin{figure}
    \begin{center}
    \includegraphics[scale=.45]{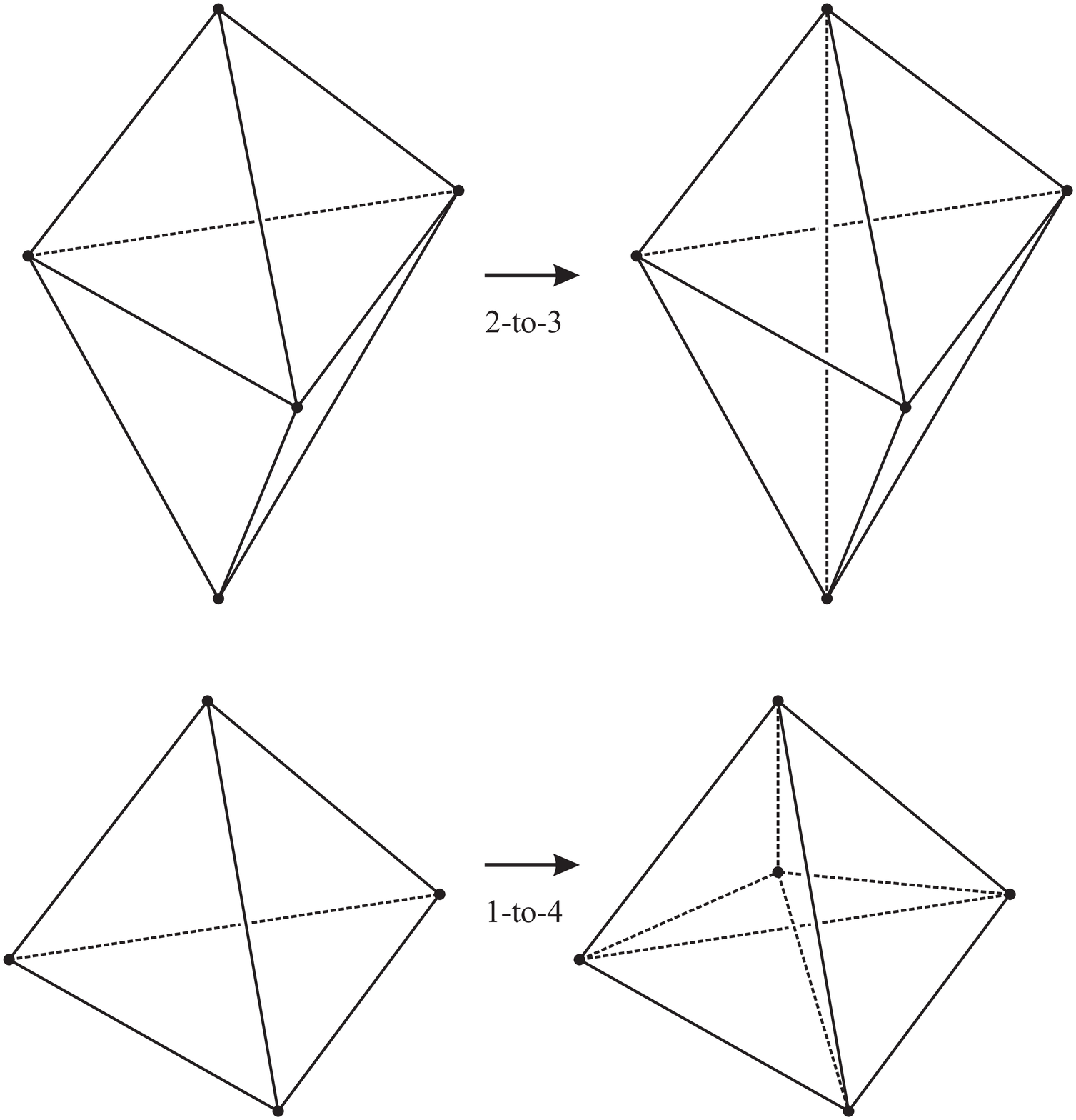}
    \mycap{The $2$-to-$3$ and the $1$-to-$4$ moves on triangulations.}
    \label{triamoves:fig}
    \end{center}
    \end{figure}

\begin{rem}\label{enum:rem}
\emph{Enumerating by computer the triangulations of closed orientable manifolds is
in principle easy, even if computationally demanding.
For increasing $n\geq 1$ one lists all the possible
orientation-reversing pairings between the faces
of $n$ tetrahedra yielding a connected result, and one checks that in the quotient
space the link of every vertex is the $2$-sphere $S^2$ (to do which one only has to show
that it has Euler characteristic $2$).}
\end{rem}

Here is a useful alternative viewpoint on triangulations. Let $M$ have one, and
consider the $2$-skeleton of the cell subdivision dual to the triangulation, as suggested
in Fig.~\ref{duality:fig}-left.
    \begin{figure}
    \begin{center}
    \includegraphics[scale=.45]{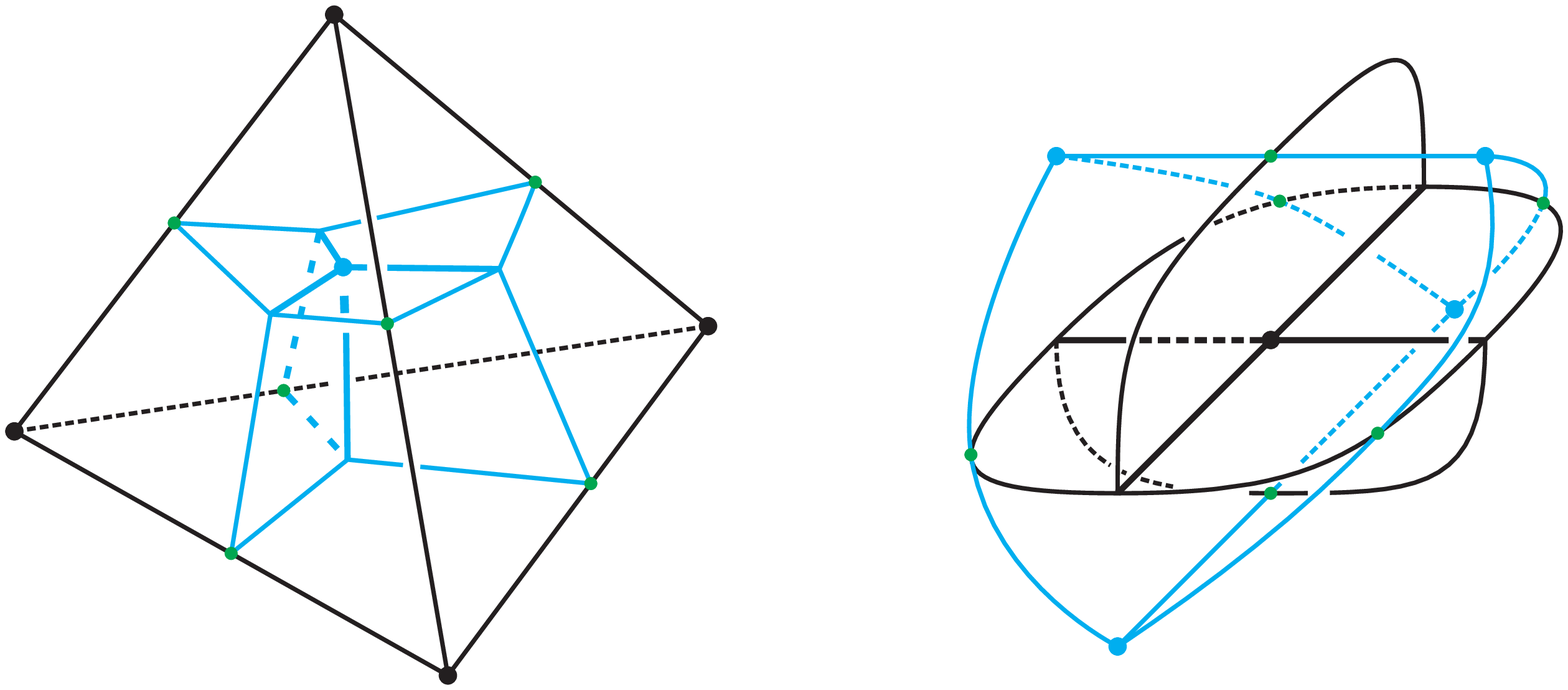}
    \mycap{Duality between triangulations and spines.}
    \label{duality:fig}
    \end{center}
    \end{figure}
This gives a \emph{spine} of $M$ minus the
vertices of the triangulation, namely a complex onto which this space collapses.
This complex is actually a \emph{special polyhedron}, namely one satisfying the
following conditions:
\begin{itemize}
\item It consists of non-singular surface points as in Fig.~\ref{special:fig}-left,
of singular points giving triple lines as in Fig.~\ref{special:fig}-center, and of
at least one singular \emph{vertex} as in Fig.~\ref{special:fig}-right;
\item The connected components of the set of non-singular points are open discs.
\end{itemize}
    \begin{figure}
    \begin{center}
    \includegraphics[scale=.45]{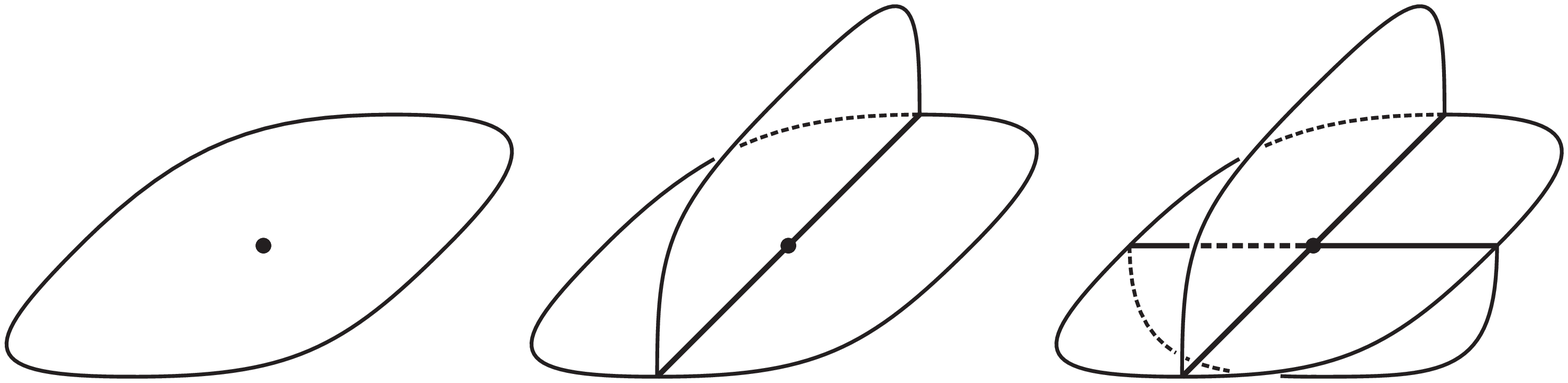}
    \mycap{Special polyhedra.}
    \label{special:fig}
    \end{center}
    \end{figure}
The construction can be reversed: using a technical
notion of \emph{orientability} for a special polyhedron (see for instance~\cite{BePe:ogra})
one uses Fig.~\ref{duality:fig}-right to associate to an orientable special polyhedron
a set of tetrahedra and a pairing between their faces. As illustrated below, this
does not always give a triangulation of a closed manifold, but one can check
whether it does along the lines of Remark~\ref{enum:rem}.
The spine versions of the moves on triangulations are shown in Fig.~\ref{spinemoves:fig}.
    \begin{figure}
    \begin{center}
    \includegraphics[scale=.45]{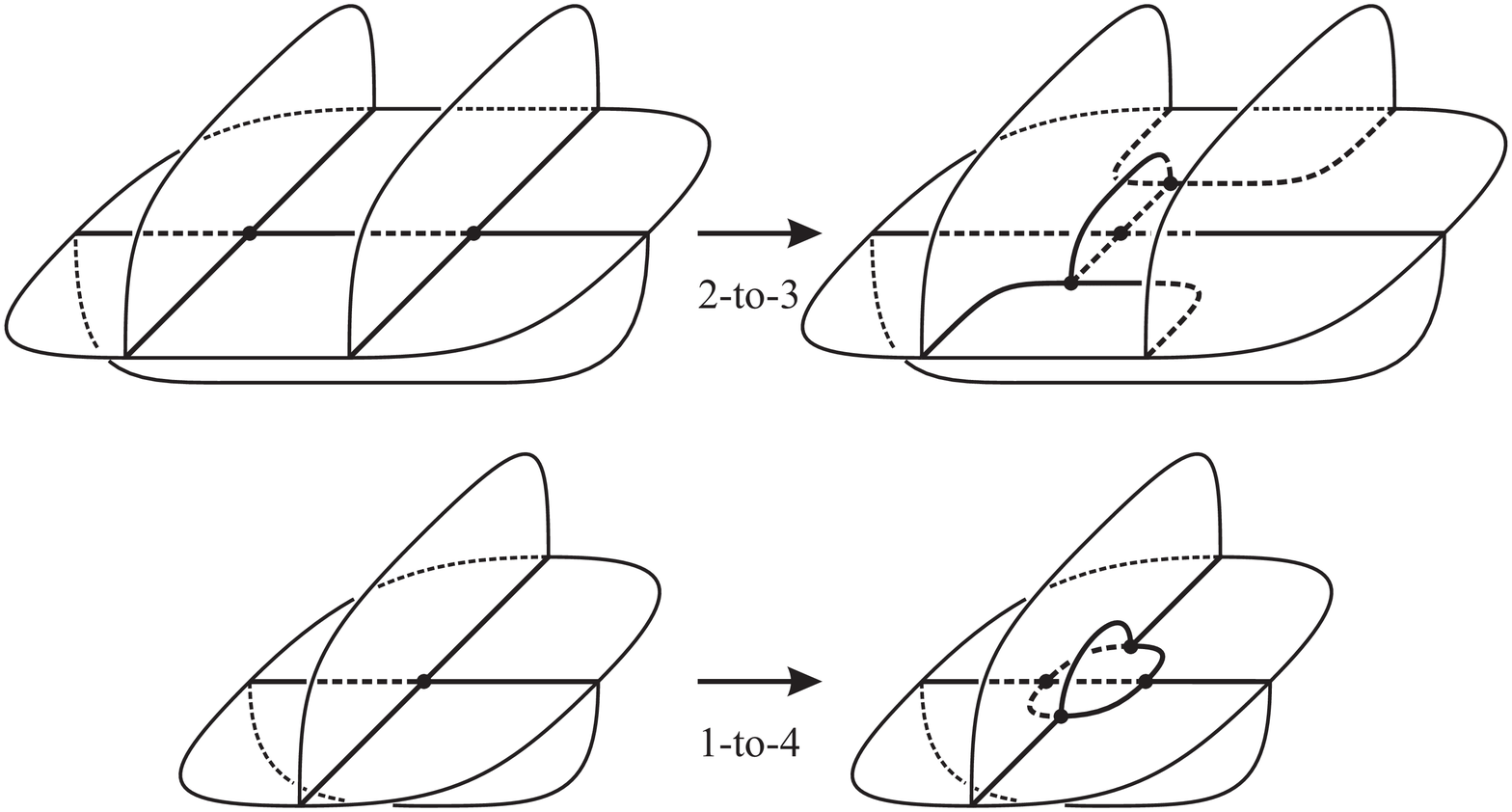}
    \mycap{The $2$-to-$3$ and the $1$-to-$4$ moves on special spines.}
    \label{spinemoves:fig}
    \end{center}
    \end{figure}

\paragraph{Ideal triangulations}
Turning to the case of a compact manifold $M$ with non-empty boundary $\partial M$,
one can adapt to $M$ the notion of (loose) triangulation by calling
\emph{ideal triangulation} any of the following pairwise equivalent notions:
\begin{itemize}
\item A realization of $M$ minus its boundary as the space obtained
by first gluing a finite number of disjoint tetrahedra along simplicial maps,
and then removing the vertices;
\item A realization of the space $X$ obtained from $M$ by collapsing each
component of $\partial M$ to a point as the quotient of a disjoint
union of tetrahedra under a simplicial pairing of the faces, in such a way that
the quotient vertices correspond to the collapsed components of $\partial M$;
\item A realization of $M$ as a gluing of truncated tetrahedra as in Fig.~\ref{trunca:fig},
with gluings between the lateral hexagons induced by simplicial gluings of the non-truncated
tetrahedra.
\end{itemize}
    \begin{figure}
    \begin{center}
    \includegraphics[scale=.45]{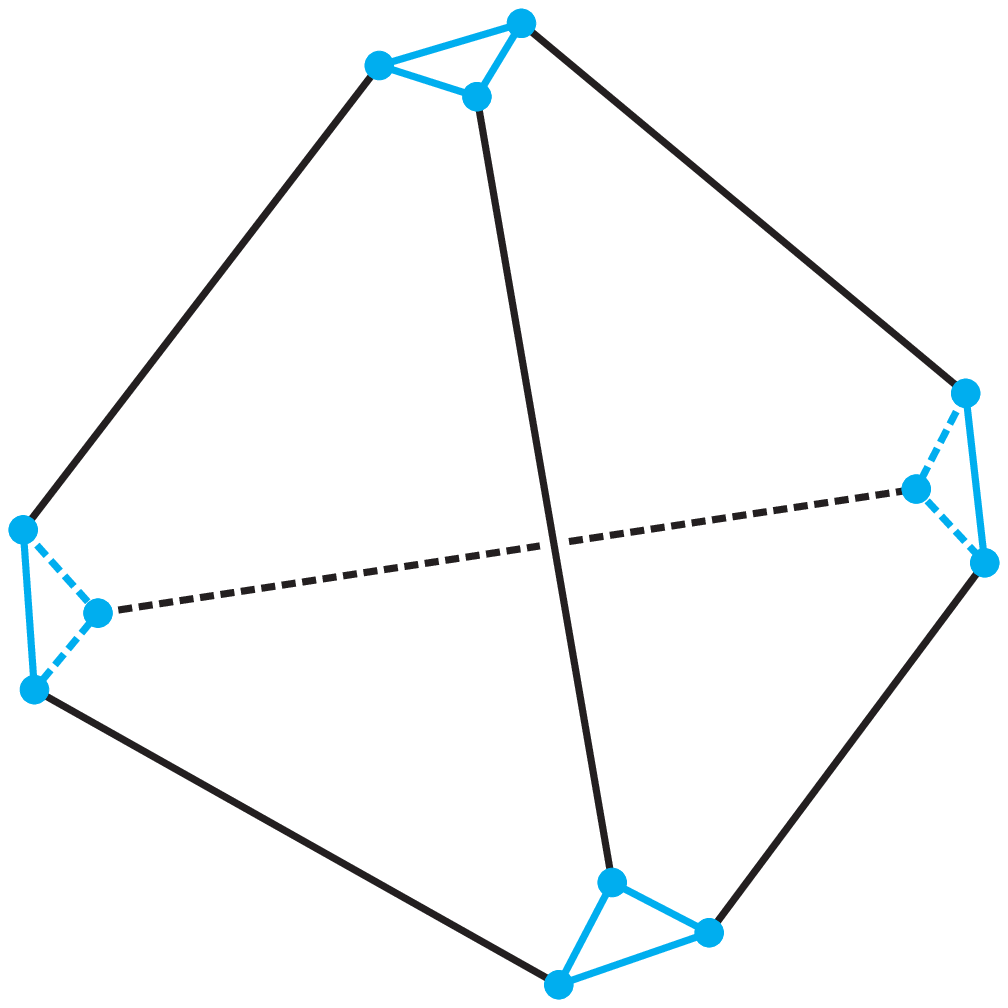}
    \mycap{A truncated tetrahedron.}
    \label{trunca:fig}
    \end{center}
    \end{figure}

For the next result we refer again to~\cite{mafo}:

\begin{thm}\label{bounded:calculus}
Any compact orientable $3$-manifold $M$ with non-empty boundary admits ideal triangulations, and
any two of them consisting of at least two tetrahedra can be transformed into each other
by repeated applications of the $2$-to-$3$ move shown in Fig.~\ref{triamoves:fig}-top
and its inverse.
\end{thm}

\begin{rem}
\emph{It is actually quite easy to deduce Theorem~\ref{closed:calculus} from
Theorem~\ref{bounded:calculus}. One only needs to remark that removing some number
$v$ of open $3$-balls from a connected and closed $M$ is a well-defined operation, from the result
of which $M$ can be reconstructed unambiguously by capping off
the boundary spheres. Moreover the $1$-to-$4$ move of
Fig.~\ref{triamoves:fig}-bottom is one that allows to increase by $1$ the number
of vertices of an ideal triangulation, and hence to increase by $1$ the number of punctures
in a punctured closed manifold represented by the triangulation.}
\end{rem}

The dual viewpoint of special spines carries over to the case of manifolds with boundary,
and the corresponding statement is actually even more expressive:

\begin{thm}\label{bounded:spines:calculus}
\begin{itemize}
\item Each orientable compact $3$-manifold with non-empty boundary admits special spines;
\item Each orientable special polyhedron is the spine of a unique $3$-manifold with non-empty boundary;
\item Two special spines of the same $3$-manifold with non-empty boundary, both having at least two vertices,
are related to each other by repeated applications of the $2$-to-$3$ move of Fig.~\ref{spinemoves:fig}-top
and its inverse.
\end{itemize}
\end{thm}

\begin{rem}\label{one:tetra:rem}
\emph{We have repeatedly excluded from our statements the triangulations consisting
of one tetrahedron only (and, dually, the spines having one vertex only). This is not
a serious issue, because only a small number of uninteresting manifolds
are described by these triangulations or spines.}
\end{rem}

\paragraph{Knots, links and graphs}
Besides manifolds, \emph{knots} are the next main objects of interest in $3$-dimensional topology.
According to the basic definition, a knot is a tamely embedded circle in $3$-space, but
one can easily extend the situation by considering \emph{links}, defined as disjoint unions
of knots, and let the ambient manifold in which a link is embedded be an arbitrary closed one.
This leads to considering pairs $(M,L)$, with closed $M$ and $L\subset M$ a link, that
we will always view up to equivalence of pairs (in the appropriate category) without further mention.
We then define a \emph{triangulation} of a link-pair $(M,L)$ as a (loose) triangulation of
$M$ that contains $L$ as a subset of its $1$-skeleton.
The next result was implicit in
the work of Turaev and Viro~\cite{TV} and was formally established by Amendola~\cite{Amen:calc}
(see also Pervova and the author~\cite{PePe:links} for more on spines of link-pairs):

\begin{thm}\label{link:calc}
Every link-pair $(M,L)$ with non-empty $L$ admits triangulations
with precisely one vertex on each component of $L$. Any two such triangulations
of $(M,L)$ consisting of at least two tetrahedra can be transformed into each other
by repeated applications of the $2$-to-$3$ move shown in Fig.~\ref{triamoves:fig}-top, and
of the inverse of this move applied when the edge that disappears with the move does not belong to $L$.
\end{thm}

A further category of objects that one deals with is given by the pairs $(M,G)$ where $M$ is a closed
$3$-manifold and $G\subset M$ is a graph, that is a $1$-subcomplex of $M$. A \emph{triangulation}
of $(M,G)$ is one of $M$ that contains $G$ as a subcomplex of its $1$-skeleton.
The previous result holds also for these objects, with the requirement that the triangulation
should have one vertex at each vertex of $G$ and one on each knot component of $G$.

\paragraph{Orbifolds}
We finally introduce orbifolds, defined as spaces having a singular differentiable structure
locally defined as the quotient of Euclidean space under the action of a finite group
of orientation-preserving diffeomorphisms. Since a finite orientable differentiable action is conjugate
to a special orthogonal one, one sees that the local group acting can be assumed to
be either cyclic, or dihedral, or the automorphism group of one
of the Platonic solids. This implies that the support of a (closed, orientable, locally orientable) $3$-orbifold
is a closed orientable $3$-manifold, in which the singular locus is a trivalent graph
with edges labelled by integers and local aspect as in Fig.~\ref{orbs:fig}.
    \begin{figure}
    \begin{center}
    \includegraphics[scale=.45]{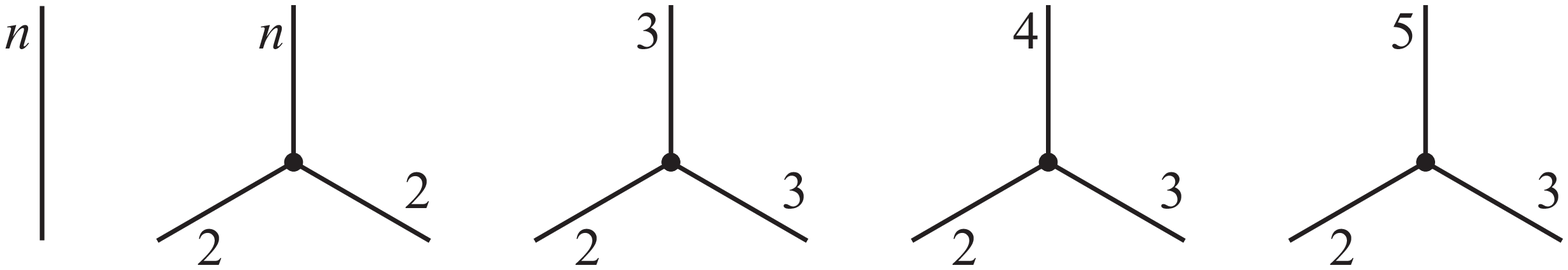}
    \mycap{Local aspect of a closed locally orientable $3$-orbifold.}
    \label{orbs:fig}
    \end{center}
    \end{figure}

\section{Hyperbolic structures}
In this section we review the definition of hyperbolic $n$-space, we summarize its main features, and
we define the hyperbolic structures we will be interested in constructing on each of the types of
topological $3$-dimensional objects illustrated in the previous section.

\paragraph{Hyperbolic $n$-space}
The $n$-dimensional hyperbolic space $\matH^n$ can be defined as the only complete and simply connected
Riemannian $n$-manifold having all sectional curvatures equal to $-1$, see~\cite{unique:hyp}.
For our purposes it will however be helpful to have at hand the following concrete models of this space:
\begin{itemize}
\item \emph{The disc model}, defined as the open unit disc $$B^n=\{x\in\matR^n:\ \|x\|<1\}$$ endowed with the metric
$$\diff s^2_x=\frac{\diff x^2}{4\left(1-\|x\|^2\right)^2};$$
\item \emph{The half-space model}, defined as the upper half-space $$\pi^n_+=\{x\in\matR^n:\ x_n>0\}$$
endowed with the metric $$\diff s^2_x=\frac{\diff x^2}{x_n^2};$$
\item \emph{The hyperboloid model}, defined as the hyperboloid
$$\calH^n_+=\left\{x\in\matR^{1,n}:\ \scalgen xx=-1,\ x_0>0\right\},$$
where $\matR^{1,n}$ is the Minkowski space $\matR^{n+1}$ endowed with the metric
$\scalgen xy=-x_0y_0+x_1y_1+\ldots+x_ny_n$; the Riemannian metric on $\calH^n_+$ is given by the
the restriction of the metric $\scalgen{\cdot}{\cdot}$ to the hyperplanes
tangent to $\calH^n_+$, on which $\scalgen{\cdot}{\cdot}$ is positive-definite.
\end{itemize}
The different models allow to single out some of the features of $\matH^n$ that we will need below
(see~\cite{bible,lectures,ratcliffe}):
\begin{itemize}
\item As one sees very well from the disc model, $\matH^n$ has a natural compactification
obtained by adding the points at infinity, that constitute an $(n-1)$-dimensional sphere $\partial\matH^n$;
\item The \emph{geodesics} of $\matH^n$ ending at the point $\infty$ in the half-space
model $\pi_n^+$ are the vertical half-lines;
\item A \emph{horosphere}, defined as a connected complete hypersurface orthogonal to all the geodesics
ending at a given point of $\partial\matH^n$, called its center, if centered at $\infty$ in the
$\pi^n_+$ model is given by a horizontal hyperplane, so it is endowed with a natural Euclidean structure;
moreover the horosphere together with its center bound a topological disc
in the compactified hyperbolic space, called a \emph{horoball};
\item An isometry $\gamma$ of $\matH^n$ must have fixed points either in $\matH^n$ or in $\partial\matH^n$, and
hence it must be of one of the following types:
\begin{itemize}
\item \emph{elliptic}, namely with fixed points in $\matH^n$; in this case,
assuming $0$ is fixed in the disc model, $\gamma$ can
be identified to an orthogonal matrix;
\item \emph{parabolic}, namely with no fixed points in $\matH^n$ and
exactly one on $\partial\matH^n$; in this case,
assuming $\infty$ is fixed in the half-space model,
$\gamma$ can be identified to an affine isometry of Euclidean space $\matR^{n-1}$
acting horizontally on $\pi^n_+$; in particular, if $n=3$ and $\gamma$ preserves the orientation,
it is just a horizontal translation;
\item \emph{hyperbolic}, namely with no fixed points in $\matH^n$ and exactly two on $\partial\matH^n$; in this case,
assuming $0$ and $\infty$ are fixed in $\pi^n_+$, it has the form
$$\gamma(x)=\lambda\cdot\left(\begin{array}{cc}A & 0 \\ 0 & 1\end{array}\right)\cdot x$$
with $A\in\mathrm{O}(n-1)$ and $\lambda>1$.
\end{itemize}
\end{itemize}

\paragraph{Closed and cusped hyperbolic manifolds}
Let us temporarily drop our assumption that all manifolds should be compact, and
take a possibly open $n$-dimensional one $N$. A \emph{hyperbolic structure} on $N$ can be defined
in one of the following equivalent ways:
\begin{itemize}
\item A complete Riemannian metric on $N$ with all sectional curvatures equal to $-1$;
\item A complete Riemannian metric on $N$ making it locally isometric to $\matH^n$;
\item An identification between $N$ and the quotient of $\matH^n$ under the action
of a discrete and torsion-free group of isometries;
\item A faithful representation of $\pi_1(N)$ into the group of the isometries of $\matH^n$ having
discrete and torsion-free image.
\end{itemize}
To state the first main general result we need to introduce further notation.
Given a Riemannian manifold $N$ and $\varepsilon>0$, we define
the $\varepsilon$-thick part $N_{[\varepsilon,+\infty)}$ of $N$ as the set of
$x\in N$ such that every loop based at $x$ and having length at most $\varepsilon$
is null in $\pi_1(N,x)$, and the $\varepsilon$-thin part
$N_{(0,\varepsilon]}$ of $N$ as the closure of the complement of its $\varepsilon$-thick part.
The following holds true:

\begin{thm}[Margulis lemma]
There exists $\varepsilon>0$ depending only on $n$ such that
if a hyperbolic $N$ is non-compact but has finite volume then its
$\varepsilon$-thick part $N_{[\varepsilon,+\infty)}$ is compact, and its $\varepsilon$-thin part
$N_{(0,\varepsilon]}$ is a disjoint union of components of the form $\Sigma\times[0,\infty)$, with $\Sigma$
a closed Euclidean $(n-1)$-manifold.
\end{thm}

Since the only closed orientable surface carrying a Euclidean structure is the torus $T$,
this result implies that an orientable $3$-dimensional finite-volume hyperbolic $N$
is the union of a compact manifold $M$ bounded by tori and a finite number of \emph{cusps} based
on tori, as suggested in Fig.~\ref{cusps:fig}.
    \begin{figure}
    \begin{center}
    \includegraphics[scale=.45]{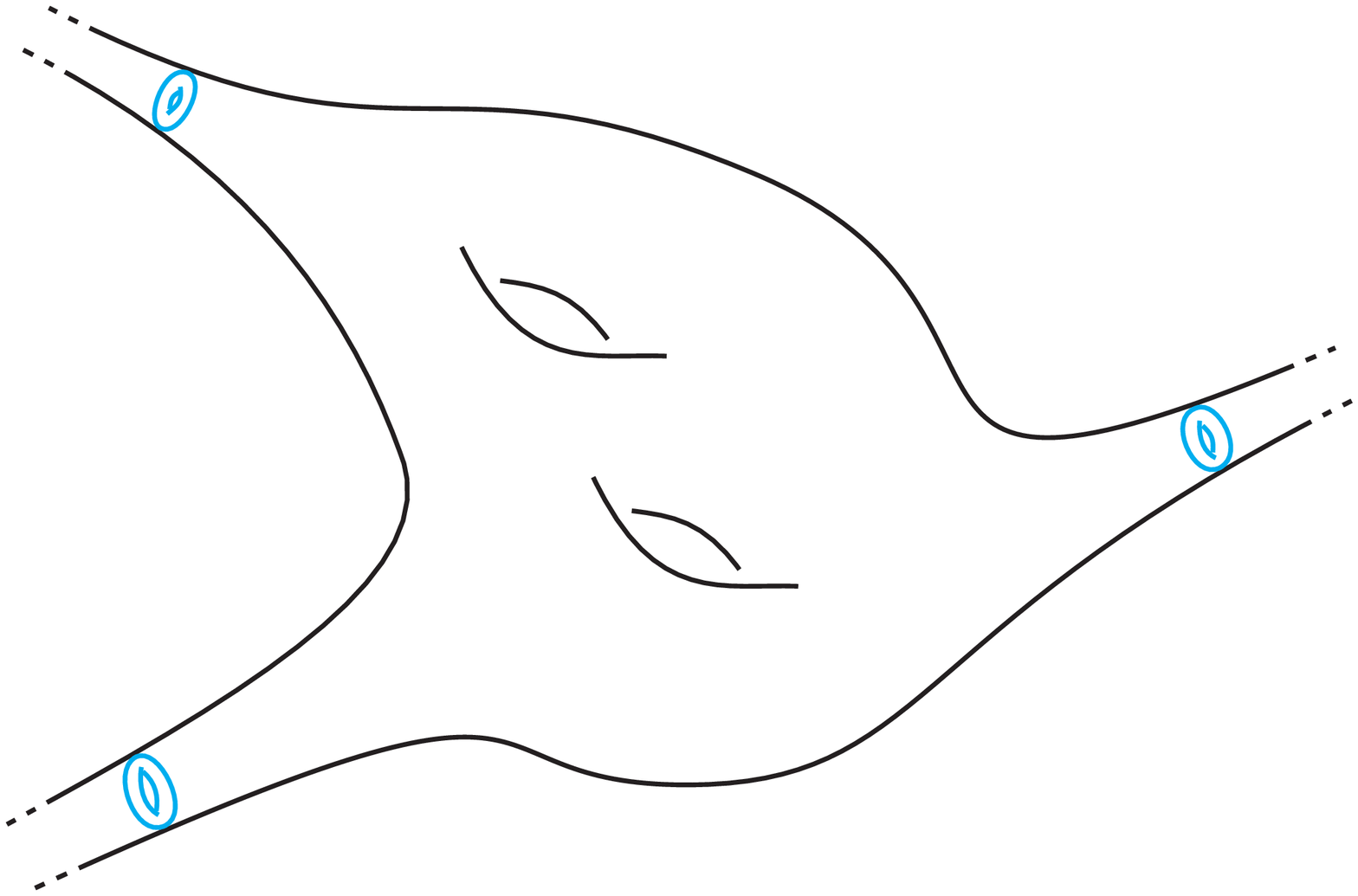}
    \mycap{An allusive picture of a cusped hyperbolic $3$-manifold.}
    \label{cusps:fig}
    \end{center}
    \end{figure}
Moreover $N$ can be identified to the interior of $M$. For this reason, with a slight abuse of terminology,
we will say that $M$ itself is hyperbolic, always meaning that the hyperbolic structure is
actually defined on the interior of $M$, and that the toric boundary components of $M$ give rise to cusps.

The next general result is the following one:

\begin{thm}[Mostow rigidity]
If $n\geq 3$ two finite-volume hyperbolic $n$-manifolds having isomorphic fundamental
groups are isometric to each other. In particular, every $n$-manifold carries at most one
finite-volume hyperbolic metric up to isometry.
\end{thm}
This deep theorem has the important consequence that any geometric invariant of a hyperbolic
manifold, such as the volume or the length of the shortest geodesic for a closed one,
is automatically a topological invariant. To state the next result, we need to recall
that performing a \emph{Dehn filling} of a torus boundary component $T$ of a compact
$3$-manifold $M$ consists in gluing to $M$ the solid torus $D^2\times S^1$ along
a homeomorphism $\partial(D^2\times S^1)\to T$. The result of this operation depends
only on the \emph{slope} on $T$ that becomes contractible in the attached
$D^2\times S^1$, namely on the isotopy class on $T$ of the simple non-trivial curve $f(S^1\times\{*\})$.
If $M$ has several boundary component we will call \emph{Dehn filling} of $M$ any manifold
obtained by performing this operation on some (possibly all) of the toric components of $\partial M$.
The next general result shows that in dimension three, given a cusped hyperbolic manifold,
one can produce a wealth of new ones:

\begin{thm}[Thurston's hyperbolic Dehn filling]\label{filling:thm}
Let $M$ be a finite-volume hyperbolic $3$-manifold with cusps based on tori $T_1,\ldots,T_k$.
Then for $j=1,\ldots,k$ there exists a set finite $E_j$ of slopes on $T_j$ such that
every Dehn filling of $M$ performed along slopes $\alpha_1,\ldots,\alpha_k$ with
$\alpha_j\not\in E_j$ is hyperbolic.
\end{thm}

Note that the theorem includes the case of the ``empty'' filling of some cusp (or several ones),
that leaves the cusp as is. We also remark in passing that one can define a
natural topology on the space of hyperbolic manifolds and that taking a sequence
of fillings of $M$ in which on each cusp the length of the slope (defined for
instance as the norm of its coordinates with respect to some fixed homological basis)
tends to infinity, one gets a sequence of hyperbolic manifolds converging to $M$,
with volumes converging from below to that of $M$.

\paragraph{Hyperbolic manifolds with geodesic boundary}
When a compact $3$-manifold $M$ has boundary components which are not tori,
one has no hope to construct on it or on its interior a finite-volume hyperbolic
structure (an infinite-volume non-rigid one often exists, but this is a
completely different story). In this case one allows the boundary of $M$ to be
part of the hyperbolic structure, in the form of a totally geodesic surface.
To explain the matter in detail, we again temporarily remove the restriction
that manifolds should be compact, and consider an arbitrary one $N$,
possibly non-compact and with boundary, with the boundary itself possibly non-compact.
We then say that $N$ is \emph{hyperbolic with geodesic boundary} if it has a complete finite-volume Riemannian
structure locally modeled on open subsets of a half-space in hyperbolic space $\matH^3$.
Mirroring $N$ in its boundary we get the double $D(N)$ of $N$, which is hyperbolic
without boundary, so its universal cover can be identified to $\matH^3$.
Moreover $\partial N$ is a totally geodesic surface in $D(N)$, and
the universal cover of $N$ can be identified to the closure of any connected component in $\matH^3$
of the complement of the
family of disjoint planes in $\matH^3$ that project in $D(N)$ onto $\partial N$.
This allows the following alternative description of a hyperbolic structure with geodesic boundary:
\begin{itemize}
\item A hyperbolic structure with geodesic boundary on $N$ corresponds to a realization of $N$
as the quotient of the intersection $H$ of a family of half-spaces in $\matH^3$
under the action of a discrete and torsion-free group of isometries of $\matH^3$ that leave $H$ invariant.
\end{itemize}

Let us now describe the thin part of a finite-volume hyperbolic $3$-manifold $N$ with geodesic
boundary. Since $D(N)$ is finite-volume hyperbolic without boundary, for $\varepsilon$ less than the third
Margulis constant the $\varepsilon$-thin part of $D(N)$ consists of cusps based on tori.
Each such cusp is either disjoint from $\partial N$, in which case it gives rise to a toric cusp
in $N$, or it is cut into two symmetric pieces by $\partial N$. It is then not too
difficult to see that the corresponding portion of the thin part of $N$ is an \emph{annular cusp},
namely of type $A\times [0,+\infty)$, with $A$ a Euclidean annulus obtained by gluing together
two opposite sides of a rectangle.

This discussion implies that a finite-volume hyperbolic $3$-manifold $N$ with geodesic boundary compactifies to a certain $M$
with a specified family of closed annuli $\calA$ on $\partial M$, so that $N$ is given by $M$ minus
$\calA$ and the toric components of $\partial M$. Note that $\partial M$ cannot contain spheres and no annulus
in $\calA$ can lie on a toric component of $\partial M$.

In the sequel we will sometimes
speak with a slight abuse of a hyperbolic compact $(M,\calA)$ to mean that a
(complete and finite-volume, as always) hyperbolic metric is defined on $M$ minus the union
of $\calA$ and all the toric boundary components of $\partial M$.

Hyperbolic structures with geodesic boundary still enjoy Mostow rigidity, but only in the sense
that each manifold can carry at most one such structure up to isometry: it is not true
in this context that the fundamental group determines the structure, as shown by Frigerio~\cite{Frigerio:nonrig}.

\paragraph{Links, orbifolds, and graphs}
For a link-pair $(M,L)$ with closed $M$ a hyperbolic structure is simply one on the exterior
of $L$ in $M$, with one cusp for each component of $L$.

Turning to a $3$-orbifold, recall that the
finite local action on $\matR^3$
defining it can be assumed to be orthogonal, up to conjugation, and that
the stabilizer of a point in the group of isometries
of hyperbolic space is the orthogonal group. The notion of a hyperbolic structure on a closed
$3$-orbifold is then an obvious extension of those already defined:
it is a complete finite-volume
singular Riemannian metric locally given by the quotient of an open ball in $\matH^3$ under
a finite action of isometries fixing the center of the ball.
Versions of the definition for orbifolds with cusps and/or with boundary exist but will
not be referred to below.

For a graph-pair $(M,G)$ we will consider
three different types of hyperbolic structure:
\begin{itemize}
\item With \emph{totally geodesic boundary}: an ordinary hyperbolic structure on the exterior $X$ of $G$ in $M$;
note that the knot components of $G$ give rise to toric cusps, whereas components with vertices give compact components of the boundary;
\item Of \emph{orbifold type}: an orbifold hyperbolic structure on $M$ with some admissible
labelling of the edges of $G$ by integers;
\item With \emph{parabolic meridians}: a hyperbolic structure on $(X,\calA)$, where $X$ is the exterior of $G$ in $M$
and $\calA$ is a system of meridinal annuli of the edges of $G$; note that for such a structure there is one
toric cusp for each component of $G$, one annular cusp for each edge joining two vertices
(or a vertex to itself), and
one thrice punctured sphere of geodesic boundary for each vertex of $G$.
\end{itemize}

\paragraph{Hyperbolisation}
So far we have not explain for what reason one should hope a $3$-dimensional manifold
(or graph, or orbifold) to have a hyperbolic structure. We now discuss the obstructions
to the existence of such a structure and state the extremely deep results according to which
the absence of these obstructions is actually sufficient to guarantee hyperbolicity.
To begin, we recall that an \emph{essential surface} in a $3$-manifold $M$ is
a properly embedded one whose fundamental group, under the inclusion,
injects into that of $M$, and which
is not parallel to the boundary. It is not too difficult to show that a hyperbolic manifold cannot
contain essential surfaces with non-negative Euler characteristic (that is, spheres, tori,
discs, or annuli). The following result has first been proved by Thurston~\cite{Haken:hyp}
for Haken manifolds (those containing some essential surface), remained as a conjecture for a long time,
and was eventually established by Perelman~\cite{perelman1,perelman2,perelman3} (see also~\cite{BBMMP}):

\begin{thm}
If a compact $3$-manifold $M$ with (possibly empty) boundary consisting of tori
does not contain any essential surface with non-negative Euler characteristic then $M$
is either hyperbolic or a Dehn filling of $P\times S^1$, where $P$ is the $2$-sphere minus
three open discs.
\end{thm}

(The reason why Dehn fillings of $P\times S^1$ make an exception is that
they are the only manifolds containing a $\pi_1$-injective \emph{immersed} torus
but no \emph{embedded} essential one, thanks to a result of Casson and Jungreis~\cite{CaJu}.)

The philosophy underlying the previous theorem is that cutting a manifold along a
surface with non-negative Euler characteristic one gets a (possibly disconnected) simpler one,
from which the original manifold can be reconstructed. Therefore hyperbolic manifolds and
Dehn fillings of $P\times S^1$ can be viewed as building blocks for general $3$-manifolds.

Hyperbolization holds true, with the necessary adjustments, for manifolds with more general
boundary (and annuli on this boundary), see~\cite{fp}, and for orbifolds
(which requires in particular the introduction of the notion of essential $2$-suborbifold), see~\cite{BoLePo,CoHoKe}.

An important consequence of the hyperbolization theorem is that if a graph-pair $(M,G)$ admits a hyperbolic structure
with totally geodesic boundary on its exterior
then for any admissible labelling of the edges, which turns $(M,G)$ into an orbifold, $(M,G)$ admits
a corresponding orbifold hyperbolic structure, and that if for some labelling of the edges
$(M,G)$ admits an orbifold hyperbolic structure then it admits one with parabolic meridians.

\section{Cusped manifolds}
We will now describe the algorithmic approach to the construction and recognition
of cusped hyperbolic manifolds, carried out with extreme success by
Callahan, Hildebrandt and Weeks~\cite{CaHiWe}.

\paragraph{Hyperbolic ideal tetrahedra}
Let us start from a compact $3$-manifold $M$ with non-empty boundary consisting of tori,
and from an ideal triangulation $\calT$ of $M$. The idea to hyperbolize $M$, which dates back to Thurston~\cite{bible},
is to choose a hyperbolic shape separately for each tetrahedron in $\calT$ and then to ensure
consistency and completeness of the structure induced on $M$. To spell out this idea we begin by
defining a \emph{hyperbolic ideal
tetrahedron} as the convex envelope $\Delta$ in $\matH^3$ of four non-aligned points in $\partial\matH^3$, endowed
with the orientation induced by $\matH^3$.
(Recall that three points on $\partial\matH^3=\matP^1(\matC)$ are always aligned, namely
there exists a geodesic plane having all three of them as points at infinity.)
Intersecting $\Delta$ with a small enough horosphere centered at any of
its vertices, one gets a Euclidean triangle, which gets rescaled if the horosphere is shrunk.
Moreover one can see that two triangles lying on horospheres centered at distinct vertices have the
same angle at the edge of $\Delta$ joining these vertices, which implies that the four triangles
at the vertices of $\Delta$ are actually similar to each other, so $\Delta$ determines a similarity class
of an oriented triangle in the plane, and the converse is also true.

To be more specific, let us note that the oriented isometries of $\matH^3$ act in a triply transitive way
on $\partial\matH^3$, so without loss of generality we can assume in the half-space model $\pi^3_+$
viewed as $(0,+\infty)\times\matC$,
that a positively oriented triple of vertices of $\Delta$ is $(0,1,\infty)$.
This implies that the fourth vertex is some $z\in\matC$ with $\Im(z)>0$, namely $z\in\pi^2_+$.
Then the hyperbolic structure of $\Delta$ is determined by $z$, that we will call \emph{module} of
$\Delta$ along the edge $(0,\infty)$, see Fig.~\ref{hyptetra:fig}.
    \begin{figure}
    \begin{center}
    \includegraphics[scale=.45]{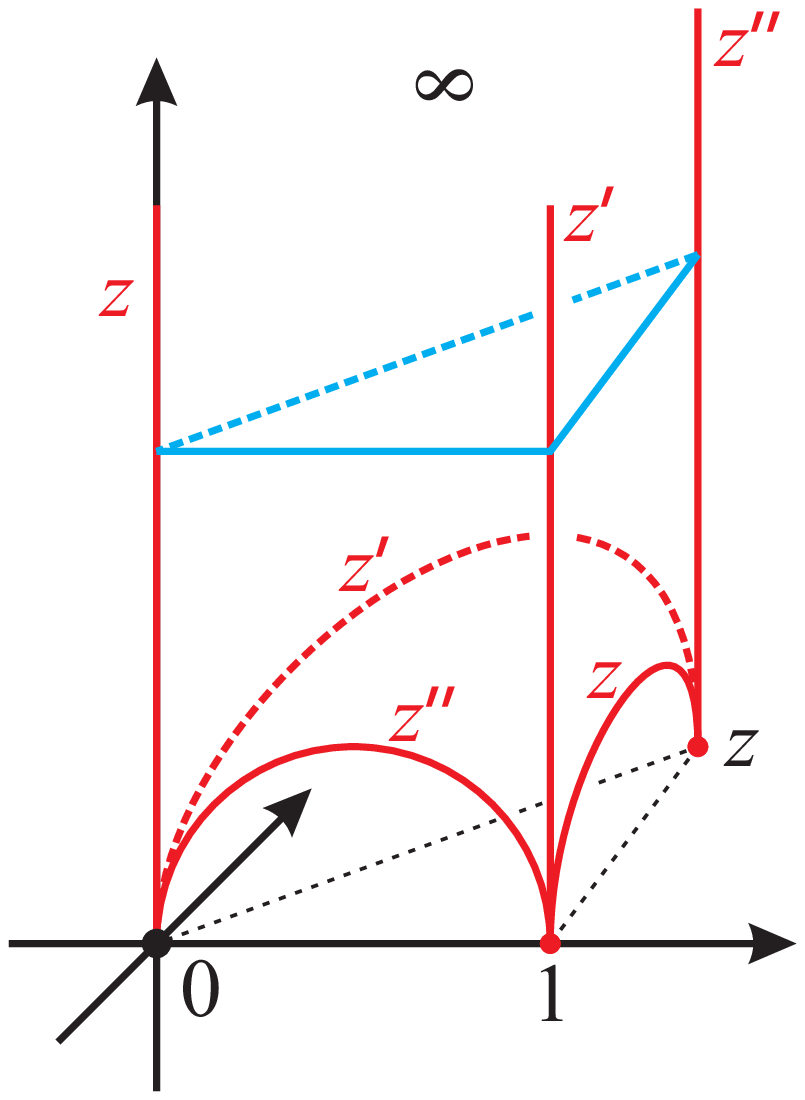}
    \mycap{Modules of a hyperbolic ideal tetrehedron.}
    \label{hyptetra:fig}
    \end{center}
    \end{figure}
Moreover the modules of $\Delta$ along the other edges are as shown in the figure,
with $z'=\frac1{1-z}$ and $z''=1-\frac1z$. In particular, $\Delta$ has the same module
along any two edges opposite to each other. And, converseley, once an orientation and
a pair of opposite edges have been fixed on an abstract tetrahedron, the
choice of any $z\in\pi^2_+$ turns the tetrahedron into an ideal hyperbolic one as in Fig.~\ref{hyptetra:fig}.

\paragraph{Consistency and completeness}
Let us return to our $M$ ideally triangulated by $\calT$, and assume that there are $n$ tetrahedra
$\Delta_1,\ldots,\Delta_n$ and $k$
toric boundary components $T_1,\ldots,T_k$. Choosing a hyperbolic structure on $\Delta_1,\ldots,\Delta_n$ corresponds to choosing
$z_1,\ldots,z_n\in\pi^2_+$, that we can view as variables. Using again the fact that the isometries
of $\matH^3$ act in a triply transitive way on $\partial\matH^3$, it is now easy to see that for any
choice of $z_1,\ldots,z_n$ the hyperbolic structure on the tetrahedra extends to the interior
of the glued faces in $M$. We then have the following:

\begin{prop}[Consistency equations]
The hyperbolic structure defined by $z_1,\ldots,z_n$ extends along an edge $e$ of $\calT$ in $M$
if and only if the product of all the modules of $\Delta_1,\ldots,\Delta_n$ (counted with multiplicity if
$e$ is multiply adjacent to some $\Delta_j$) equals $1$, and the sum of the arguments of these modules
equals $2\pi$.
\end{prop}

\begin{rem}
\emph{If the product of the modules along an edge equals $1$, then the sum of the
arguments of these modules is a positive multiple of $2\pi$. Using this fact
and the observation that $\chi(M)=0$, because $\partial M$ consists
of tori, one then sees that if the products of the modules along \emph{all} edges of
$\calT$ equals $1$, then the sum of the arguments always equals $2\pi$. This implies that consistency
of the hyperbolic structure defined by $z_1,\ldots,z_n$ translates into $n$ algebraic equations.
(See also below for the number of these equations.)}
\end{rem}

Suppose now that $z_1,\ldots,z_n$ satisfy the consistency equations along all the edges of $\calT$.
Then each boundary torus $T_j$ is obtained by gluing Euclidean triangles along similarities, and
consistency ensures that the similarity structure on the triangles extends to the edges and the vertices.
Summing up, $z_1,\ldots,z_n$ induce a similarity structure on each $T_j$, and we have:

\begin{prop}\label{compl:prop}
The hyperbolic structure on $M$ defined by $z_1,\ldots,z_n$ is complete if and only if
the induced similarity structure on each $T_j$ is actually Euclidean.
\end{prop}

To turn the completeness condition into equations, we note that a similarity structure on a torus
$T$ induces a representation (the \emph{holonomy}) of $\pi_1(T)$ into the group of complex-affine automorphisms of $\matC$.
This representation is well-defined up to conjugation, so its dilation component $\rho:\pi_1(T)\to\matC_*$
is well-defined, and of course $T$ is Euclidean if and only if $\rho$ is identically $1$.
If the similarity structure on $T$ is obtained by gluing triangles with specified modules,
and $\alpha$ is a simplicial loop in the resulting triangulation, one can easily
show that $\rho(\alpha)$ is the product of the modules of the triangles that $\alpha$ leaves to
its left, as suggested in Fig.~\ref{holonomy:fig}. Therefore:
    \begin{figure}
    \begin{center}
    \includegraphics[scale=.45]{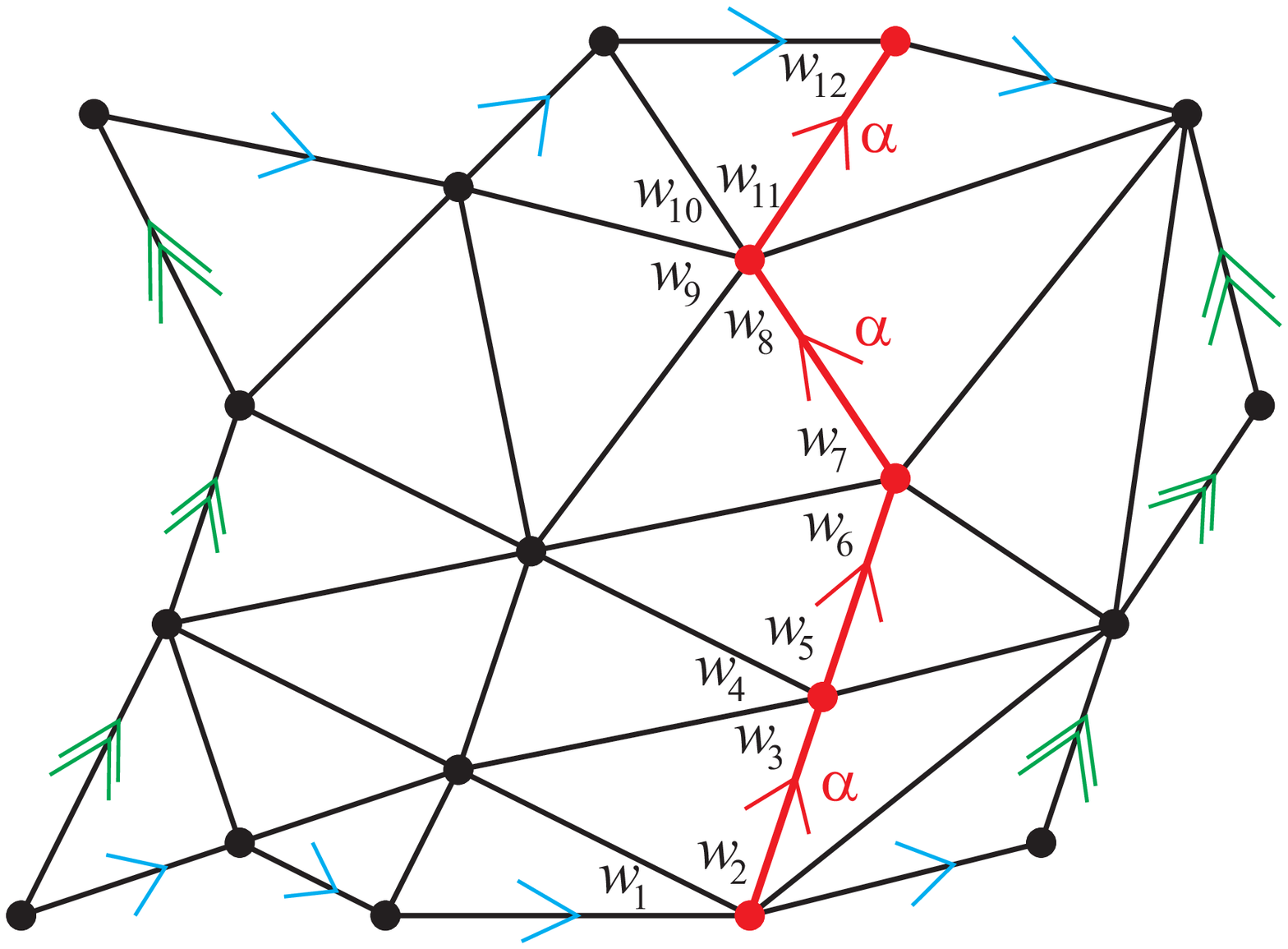}\hspace{4cm}\ $\phantom{.}$\\
    \vspace{-3cm}\hspace{8cm}\qquad $\Rightarrow\ \rho(\alpha)=w_1\cdots w_{12}$\vspace{3cm}\
    \mycap{Computation of the dilation component of the holonomy of a simplicial loop.}
    \label{holonomy:fig}
    \end{center}
    \end{figure}

\begin{prop}[Completeness equations]\label{compl:eq:prop}
For $j=1,\ldots,k$ let $\lambda_j$ and $\mu_j$ be generators of $\pi_1(T_j)$.
The hyperbolic structure on $M$ defined by $z_1,\ldots,z_n$ is complete if and only if
for all $j$ the product of the modules of the triangles on $T_j$ that $\lambda_j$ leaves to its
left equals $1$, and the same happens for $\mu_j$.
\end{prop}

\begin{rem}
\emph{The images of $\lambda_j$ and $\mu_j$ under the holonomy representation of  $\pi_1(T)$
associated to a similarity structure are commuting complex-affine automorphisms of $\matC$.
The condition $\rho(\lambda_j)=1$ means that the holonomy of $\lambda_j$ is a translation;
if this translation is non-trivial then also $\mu_j$ maps to a translation, therefore
$\rho(\mu_j)=1$. This shows that the two conditions to impose on each $T_j$ are ``almost''
equivalent to each other, so in practice one adds to the $n$ consistency equations only $k$,
and not $2k$, completeness equations. Moreover it was shown by Neumann and Zagier~\cite{NeZa} that if
a solution exists then $k$ of the consistency equations can be dismissed; moreover
the complete structure corresponds to a smooth point in the space of deformations of the
structure, which is a $k$-dimensional algebraic variety. This fact can be exploited for instance to
establish Theorem~\ref{filling:thm}.}
\end{rem}

To conclude the discussion on the construction of the hyperbolic structure on a would-be cusped manifold $M$,
we note that using an arbitrary ideal triangulation $\calT$ of $M$ it is
not true that a solution of the corresponding consistency and completeness equations always exists,
even if $M$ is actually hyperbolic. And, as a matter of fact, it is not even known
that one $\calT$ such that the corresponding equations have a genuine solution exists (despite the wrong
statement in~\cite{lectures} that this follows from~\cite{EpPe}, see also below).
However when one starts from a \emph{minimal} triangulation of a hyperbolic $M$,
namely one with a minimal number of tetrahedra, the solution always exists in practice.
Weeks' wonderful software SnapPea~\cite{SnapPea} is capable (among other things) to find a minimal triangulation of a given
(a priori possibly non-hyperbolic) $M$, to seek for a solution of the corresponding equations, and
also to deduce from patterns it sees in the triangulation the existence of topological obstructions to
hyperbolicity. It is using these features (and the recognition machinery described in the rest of this section)
that the census~\cite{CaHiWe} of cusped manifolds
triangulated by at most $7$ tetrahedra has been obtained.

\paragraph{Canonical decomposition}
Once the hyperbolic structure on a cusped $M$ has been constructed, the need  naturally arises
to \emph{recognize} such an $M$, namely to be able to effectively determine whether $M$ is the same
as any other given cusped manifold. Several hyperbolic invariants, and chiefly the volume
(which is easily computed from a hyperbolic ideal triangulation by means of
the Lobachevski function, see~\cite{milnor}), can often distinguish manifolds,
but different manifolds actually can have the same volume, as proved by
Adams~\cite{Adams}, and other invariants, so the need of
a complete one remains. This complete invariant is provided by a
result of Epstein and Penner~\cite{EpPe}, and it allows to perform the recognition
very efficiently. We will first state this result informally and then provide the necessary details.

The basic underlying idea is best described starting from
an arbitrary compact Riemannian manifold $X$ (of any dimension) with non-empty boundary.
In this case one can define the cut-locus $\cut_X(\partial X)$ of $\partial X$ in $X$
as the set of points joined by more than one distance-minimizing path to $\partial X$.
To visualize $\cut_X(\partial X)$, imagine that we start pushing all the components of $\partial X$ towards
the interior of $X$, all at the same pace. At some point some collision (or self-collision)
will start occurring; we then fuse together the collided points, leave them still henceforth,
and keep pushing the rest. Eventually we exhaust all the space available in $M$ and we are
left with $\cut_X(\partial X)$ in the form of the membrane on which the collisions have taken place.
(See Fig.~\ref{cutlocus:fig}
    \begin{figure}
    \begin{center}
    \includegraphics[scale=.45]{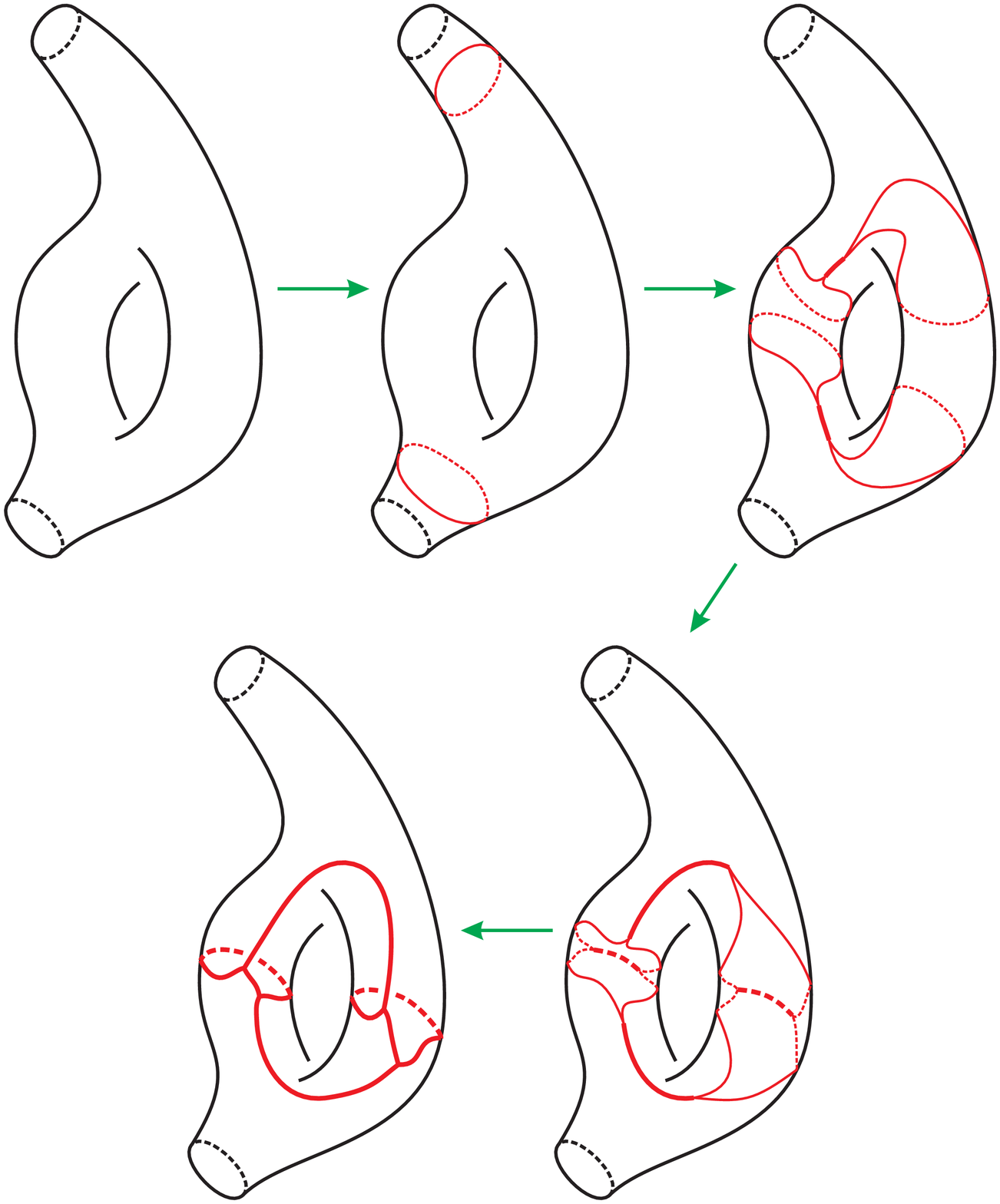}
    \mycap{The cut locus of the boundary in a Riemannian surface as the result of pushing
    the boundary towards the interior as far as possible.}
    \label{cutlocus:fig}
    \end{center}
    \end{figure}
for an allusive picture in dimension $2$.)
This description should make it obvious that $\cut_X(\partial X)$ is a compact subset of
$X$ onto which $X$ retracts, and that it has dimension at least one less than $X$.
Supposing $X$ has dimension $3$ one can in addition imagine that in a generic situation
$\cut_X(\partial X)$ will be a special spine of $X$, and therefore that dual to
it there will be a topological ideal triangulation of $X$. In more general contexts
dual to $\cut_X(\partial X)$ there will be a decomposition of $X$ into ideal polyhedra
more complicated than tetrahedra.

Turning to a cusped hyperbolic $M$, we first note that we cannot take
$\cut_M(\partial M)$, because $\partial M$ is at infinite distance from any
point in the interior of $M$, since $\partial M$ is not really part of the
hyperbolic structure, but rather of its compactification.
Recall however that each cusp of $M$ has the form $T\times[0,\infty)$,
where $T$ is a flat torus, and more precisely the image in $M$ of a horoball of $\matH^3$
acted on by the $\matZ\oplus\matZ$ lattice of the parabolic elements of $\pi_1(M)$
fixing the center of the horoball. If we replace the cusp $T\times[0,+\infty)$ with
$T\times[h,+\infty)$ for some $h>0$ we get a smaller cusp, with volume that tends to $0$
as $h\to+\infty$. Therefore for sufficiently small $v>0$ we can take disjoint
cusps at each end of $M$ all having volume $v$, and call $M^{(v)}$ the complement
in $M$ of their interior. The following fact has an intimately hyperbolic
nature, as we will explain before providing a detailed proof:

\begin{prop}\label{shrink:prop}
$\cut_{M^{(v)}}(\partial M^{(v)})$ is independent of $v$.
\end{prop}

To appreciate this result, consider the case of a Riemannian manifold $X=T\times[0,1]$,
with metric $\diff s^2_{(p,t)}=f(t)\cdot\diff\sigma^2_p+\diff t^2$,
where $\diff\sigma^2$ is a flat metric on $T$ giving it area $1$, and $f$ is a smooth incresasing function
such that $f(t)=1$ for $0\leq t\leq \frac13$ and $f(t)=2$ for $\frac56\leq t\leq 1$.
Viewing $T\times\left[0,\frac13\right]$ and $T\times\left[\frac56,1\right]$ as the ends of
$X$, we see that they both have volume $\frac13$, so $X^{\left(\frac13\right)}=T\times\left[\frac13,\frac56\right]$,
and $\cut_{X^{\left(\frac13\right)}}\left(\partial X^{\left(\frac13\right)}\right)=T\times\{\frac7{12}\}$.
However $X^{\left(\frac16\right)}=T\times\left[\frac16,\frac{11}{12}\right]$ and
$\cut_{X^{\left(\frac16\right)}}\left(\partial X^{\left(\frac16\right)}\right)=T\times\{\frac{13}{24}\}$.

\bigskip

\noindent\emph{Proof of Proposition~\ref{shrink:prop}.}
Assume two cusps of volume $v$ in $M$ get lifted in $\matH^3$ to
horoballs centered at $\infty$ and at $0$ in the $\pi^3_+$ model
of $\matH^3$, namely to some half-space
$O_1=[h_1,+\infty)\times\matR^2$ and to some Euclidean ball $O_2$ of radius $\frac{h_2}2$ centered
at $(0,0,\frac{h_2}2)$, so that its top point has height $h_2$.
Since the cusps in $M$ are disjoint or coincide, one has $h_2<h_1$.
Now suppose that the action of the $\matZ\oplus\matZ$ lattice of parabolic elements
of $\pi_1(M)$ fixing $\infty$ gives as a quotient of $\{0\}\times\matR^2$ a flat torus of area $a_1$.
Note that $a_1$ is independent of $v$, namely, if we change $v$ then the height $h_1$ changes but the
area $a_1$ does not. Moreover $v$ is equal to the integral of the
volume form $\frac1{x_3^3}\diff x_1\diff x_2\diff x_3$ of $\matH^3$ over
$[h_1,+\infty)\times A_1$, where $A_1$ is a parallelogram of area $a_1$, therefore
$v=\frac{a_1}{2h_1^2}$. Applying the inversion with respect to the radius-1 sphere
centered at $0$, which is a hyperbolic isometry, $O_2$
becomes the half-plane $\left[\frac1{h_2},+\infty\right)\times\matR^2$,
and the computation already performed shows that $v=\frac{a_2h_2^2}{2}$,
for some $a_2$ again independent of $v$. The surface of the points having equal distance from $O_1$
and from $O_2$ is of course determined by the point in which it intersects the
$x_3$-axis, whose height $h$ must satisfy
$$\int_{h_2}^h\frac{\diff x_3}{x_3}=\int_{h}^{h_1}\frac{\diff x_3}{x_3}
\ \Rightarrow\ \log(h)-\log(h_2)=\log(h_1)-\log(h)$$
$$\ \Rightarrow\ 2\log(h)=\log(h_1h_2)
\ \Rightarrow\ h=\sqrt{h_1h_2}.$$
The relations $v=\frac{a_1}{2h_1^2}$ and $v=\frac{a_2h_2^2}{2}$ established above now easily imply that
$\sqrt{h_1h_2}=\frac{a_1}{a_2}$ is indeed independent of $v$, and the conclusion follows.
\finedimo

We can now state the result of~\cite{EpPe}:

\begin{thm}[Epstein-Penner canonical decomposition]\label{EpPe:thm}
If $M$ is a cusped hyperbolic $3$-manifold then dual to
$\cut_{M^{(v)}}(\partial M^{(v)})$ there is a decomposition
of $M$ into hyperbolic ideal polyhedra whose combinatorics and
hyperbolic shape of the blocks depends on $M$ only.
\end{thm}

Once the Epstein-Penner canonical decompositions of two given cusped
hyperbolic manifolds have been determined, to compare the manifolds for equality
one then only needs to compare the canonical decompositions for combinatorial
equivalence. Note that one does not need to check that the hyperbolic shapes
of the polyhedra are the same, since combinatorial equivalence of the decompositions
already ensures that the manifolds are homeomorphic to each other (whence, by rigidity, isometric to each other).

\paragraph{The light-cone and the convex hull construction}
To show how one can actually construct the Epstein-Penner canonical
decomposition of a given ideally triangulated cusped manifold, we will
exploit more of the hyperboloid model $\calH^n_+$ of $\matH^n$ than we
have done so far. We first define the \emph{(future) light-cone}
in the Minkoswki space $\matR^{1,n}$ with scalar product $\scalgen{\cdot}{\cdot}$ as
$$\calL^n_+=\left\{y\in\matR^{1,n}:\ \scalgen{y}{y}=0,\ y_0>0\right\}$$
and we remark that there is a natural identification between
$\partial\matH^n$ and the projectivized light-cone
$\matP\left(\calL^n_+\right)$. Moreover for all $y\in\calL^n_+$ one can define as follows an associated horoball
$$B_y=\left\{x\in\calH^n_+:\ \scalgen xy\geq -1\right\}$$
and its boundary horosphere $H_y=\partial B_y$. It is not hard to see that $B_y$ is centered at
$[y]\in\matP\left(\calL^n_+\right)=\partial\matH^n$, and that all horoballs centered at some $p\in\matP\left(\calL^n_+\right)=\partial\matH^n$
have the form $B_y$ for some $y\in\calL^n_+$ with $[y]=p$. Note that $B_{y'}\subset B_y$ if
$y'=\lambda y$ with $\lambda>1$.

Turning to the effective construction of the Epstein-Penner decomposition,
let us fix a cusped hyperbolic $3$-manifold $M$, and a set of disjoint cusps
in $M$ all having one and the same volume $v$. These cusps lift in the universal cover
of $M$, that we identify with $\calH^3_+$, to a family of disjoint horoballs
$\{B_y:\ y\in\calP\}$ for some $\calP\subset\calL^3_+$.
Let us now establish the following crucial property of $\calP$:

\begin{lemma}
$\calP$ is discrete.
\end{lemma}

\begin{proof}
It is of course sufficient to show that for all $h>0$ the set $\{p\in\calP:\ x_0(p)\leq h\}$ is finite.
Assuming the contrary and projecting to the disc model $B^3$, we would get an infinite family
of horoballs that, as Euclidean balls, have radius bounded from below. But this is impossible since the horoballs
must be disjoint from each other.
\end{proof}

We now define $C$ as the convex hull of $\calP$ in $\matR^{1,3}$, and we note that $\calP$, and hence $C$,
are invariant under the action of $\pi_1(M)$, which extends from $\calH^3_+$ to $\matR^{1,3}$.
The following is established in~\cite{EpPe}:

\begin{prop}\label{hull:prop}
\begin{itemize}
\item $C\cap\calL^3_+=\{\lambda\cdot p:\ p\in\calP,\ \lambda\geq 1\}$;
\item For all $x\in\calH^3_+$ the half-line $\{t\cdot x:\ t\geq 0\}$ intersects
$C$ in a half-line $\{t\cdot x:\ t\geq \lambda_0(x)\}$ for a suitable
$\lambda_0(x)$, and $\lambda_0(x)\cdot x\in\partial C$:
\item $\partial C\setminus\calL^3_+$ consists precisely of the points
$\lambda_0(x)\cdot x$ for $x\in\calH^3_+$, therefore the radial projection
is a bijection between $\partial C\setminus\calL^3_+$ and $\calH^3_+$;
\item $\partial C$ consists of a $\pi_1(M)$-invariant family of finite-faced polyhedra that
intersect $\calL^3_+$ precisely at their vertices;
\item The polyhedra of which $\partial C$ consists, projected first radially to $\calH^3_+$ and then
to $M$ under the action of $\pi_1(M)$, give the ideal decomposition of $M$ dual to
$\cut_{M^{(v)}}(\partial M^{(v)})$ as in Theorem~\ref{EpPe:thm}
\end{itemize}
\end{prop}

\paragraph{The tilt formula}
Let us suppose that $M$ is a cusped hyperbolic manifold with a given
hyperbolic ideal triangulation $\calT$. We will now describe a method,
based on the results of Sakuma and Weeks~\cite{Jeff:tilt,SaWe} and exploited by
Weeks' software SnapPea~\cite{SnapPea}, to decide whether
the Epstein-Penner canonical decomposition of $M$ is actually
$\calT$ or can be obtained from $\calT$ by merging together some
of the tetrahedra into more complicated polyhedra. To this end we fix
some $v>0$ such that $M$ contains disjoint cusps of volume $v$ at all its ends
(and note that $v$ is easy to find using the combinatorics of $\calT$ and
the geometry of the hyperbolic tetrahedra that $\calT$ consists of). We then concentrate on
a $2$-face $F$ of $\calT$, to which two tetrahedra $\Delta_1$ and $\Delta_2$ will be incident.
Let us lift $F,\Delta_1,\Delta_2$ in $\calH^3_+$ to an ideal triangle $\widetilde{F}$ and
two ideal tetrahedra $\widetilde{\Delta}_1$ and $\widetilde{\Delta}_2$ such that
$\widetilde{F}=\widetilde{\Delta}_1\cap\widetilde{\Delta}_2$.
The choice of $v$ allows us to associate a point on the light-cone $\calL^3_+$ to each ideal
vertex of $\widetilde{F},\widetilde{\Delta}_1,\widetilde{\Delta}_2$, and we can consider
the straight triangle $F'$ and tetrahedra $\Delta'_1,\Delta'_2$
in $\matR^{1,3}$ having these points on $\calL^3_+$ as vertices .
Finally, we define $\vartheta(F)$ as the dihedral angle in $\matR^{1,3}$ not containing $0$
formed along the plane containing $F'$ by the half-hyperplanes containing $\Delta'_1$ and $\Delta'_2$,
and we note that $\vartheta(F)$ is independent of the particular liftings chosen.
The following is a direct consequence of Proposition~\ref{hull:prop}:

\begin{prop}\label{theta:prop}
$\calT$ is the Epstein-Penner decomposition of $M$ if and only if $\vartheta(F)<\pi$ for all
$2$-faces $F$ of $\calT$. More generally, the Epstein-Penner decomposition of $M$
is obtained from $\calT$ by merging together some of the tetrahedra of $\calT$ if and only
if $\vartheta(F)\leq \pi$ for all $2$-faces $F$ of $\calT$, and in this case the mergings
to perform are those along the $F$'s such that $\vartheta(F)=\pi$.
\end{prop}

When $\calT$ does not meet the conditions of this proposition, namely when it contains some offending
$2$-face $F$ with $\vartheta(F)>\pi$, a new triangulation $\calT'$ with better chances
of being a subdivision of the Epstein-Penner decomposition is obtained by performing
the 2-to-3 move along $F$. Note however that the move can be applied only if the two tetrahedra of $\calT$ incident
to $F$ are distinct. One can then start a process that searches for faces $F$ with
$\vartheta(F)>\pi$ to which the $2$-to-$3$ move can be applied, applies the move and starts over again.
The process can get stuck if all $F$'s with $\vartheta(F)>\pi$ are incident to the
same tetrahedron on both sides, but it is shown in~\cite{SaWe} that if the process does
not get stuck then it converges in finite time to a subdivision of the Epstein-Penner decomposition.
As a matter of fact, experimentally the process always converges, and it does so very quickly.

There is however one aspect of the process just described that we have not yet described
how to perform algorithmically, namely the computation of the angle $\vartheta(F)$. This is
done using the Sakuma-Weeks \emph{tilt formula}, that constitutes the core of~\cite{SaWe}.
This formula associates a real number $t(\Delta,E)$, called the \emph{tilt},
to each triangle $E$ in $\matR^{1,3}$ with vertices on $\calL^3_+$
and to each tetrahedron $\Delta$ with vertices on $\calL^3_+$ having the triangle as a face.
To do so, the following vectors are needed:
\begin{itemize}
\item The unit normal $p$ to $\Delta$, satisfying $\scal px=-1$ for all $x\in\Delta$;
\item The outer unit normal $m$ to $E$, satisfying $\scal mm=1$, $\scal mx=0$ for all $x\in E$, and $\scal mv<0$, where
$v$ is the vertex of $\Delta$ opposite to $E$.
\end{itemize}
Then $t(\Delta,E)=\scal pm$, and the knowledge of the tilts allows to
apply Proposition~\ref{theta:prop} by means of the following:

\begin{prop}
With the notation introduced before Proposition~\ref{theta:prop},
$\vartheta(F)<\pi$ if and only if
$t(\Delta_1',F')+t(\Delta_2',F')<0$, and
$\vartheta(F)=\pi$ if and only if
$t(\Delta_1',F')+t(\Delta_2',F')=0$.
\end{prop}

Rather than provinding a complete proof of this result, we motivate it using an example
in one dimension less. We suppose in $\matR^{1,2}$ that $F'$ has vertices
$\vettt110$ and $\vettt1{-1}0$, while $\Delta'_1$ has the further
vertex $\vettt{s_1}0{s_1}$ and $\Delta'_2$ has the further vertex $\vettt{s_2}0{-s_2}$, for
some $s_1,s_2>0$. Now one easily sees that $\vartheta(F)=\pi$ if and only if the vectors
$\vettt{s_1}0{s_1},\vettt100,\vettt{s_2}0{-s_2}$ are aligned,
which happens if $\frac1{s_1}+\frac1{s_2}=2$. And more generally that $\vartheta(F)<\pi$ if and only
if $\frac1{s_1}+\frac1{s_2}<2$. A direct computation now shows that
$$p_1=\vettt10{1-\frac1{s_1}},\quad p_2=\vettt10{\frac1{s_2}-1},\quad
m_1=\vettt00{-1},\quad m_2=\vettt001$$
\begin{eqnarray*}
\Rightarrow t(\Delta_1',F')+t(\Delta_2',F') & = & \scal{p_1}{m_1}+\scal{p_2}{m_2}\\
& = & \frac1{s_1}-1+\frac1{s_2}-1\\
& = & \frac1{s_1}+\frac1{s_2}-2
\end{eqnarray*}
and the conclusion easily follows.

\bigskip

The next result, which represents the main achievement of~\cite{SaWe,Jeff:tilt}, shows
how to effectively compute the tilts of a given hyperbolic ideal triangulation of $M$,
knowing only the geometry of the hyperbolic tetrahedra
of $\calT$ and the ``height'' in each of them of the horospheres giving equal-volume
cusps in $M$.

\begin{prop}
Let $\Delta\in\calT$ have vertices $w_0,w_1,w_2,w_3$ and define
$E_i$ as the face opposite to $w_i$. Let $\Delta'\subset\matR^{1,3}$
be the straight lifting of $\Delta$ with vertices on $\calL^3_+$
determined by the choice of equal-volume cusps of $M$, let $E'_i$ be the face of
$\Delta'$ projecting to $E_i$, and set
$t_i=t(\Delta',E_i')$.
Denote by $\theta_{ij}$ the dihedral angle of $\Delta$ along the edge joining
$w_i$ and $w_j$, and by $r_i$ the radius of the circle circumscribed to the
Euclidean triangle obtained by intersecting with $\Delta$ the horosphere
centered at $w_i$ that belongs to the system yielding in $M$ the equal-volume cusps.
Then:
$$\left(\begin{array}{c}t_0\\ t_1\\ t_2\\ t_3
\end{array}\right)=
\left(\begin{array}{cccc}1 & -\cos\theta_{01} & -\cos\theta_{02} & -\cos\theta_{03}\\
-\cos\theta_{01}  & 1 & -\cos\theta_{12} & -\cos\theta_{13}\\
-\cos\theta_{02} & -\cos\theta_{12} & 1 & -\cos\theta_{23}\\
-\cos\theta_{03}  & -\cos\theta_{13} & -\cos\theta_{23} & 1
\end{array}\right)\cdot
\left(\begin{array}{c}r_0\\ r_1\\ r_2\\ r_3
\end{array}\right).$$
Moreover $r_i=e^{-d_i}$ where $d_i$ is the (signed) distance between $E_i$ and the horosphere at $w_i$
already described.
\end{prop}

We note that this result was first established in~\cite{Jeff:tilt} in dimension $3$ and then
generalized in~\cite{SaWe} for all dimensions. Giving a complete proof is beyond our scopes, but
we can at least prove the last assertion in dimension $2$. If in $\pi^2_+\subset\matC$ the triangle
$\Delta$ has vertices $0,2,\infty$ and the horosphere centered at $\infty$ cuts it at height
$\Im(z)=h$, therefore in a segment of Euclidean length $\frac2h$, then of course $r_\infty=\frac1h$.
The point closest to $\infty$ of the edge opposite to $\infty$ is $1+i$, and its distance from
the horosphere is $d_\infty=\int\limits_1^h\frac{\diff s}{s}=\log(h)$, therefore
one indeed has $r_\infty=e^{-d_\infty}$. This concludes our discussion
of the algorithmic recognition of cusped manifolds.

\paragraph{Existence of hyperbolic ideal triangulations}
Taking a subdivision of the Epstein-Penner decomposition into
hyperbolic ideal polyhedra of a cusped hyperbolic manifold $M$, one
can get a \emph{topological} ideal triangulation,
which in $M$ gives a hyperbolic triangulation with some genuine
and some \emph{flat} tetrahedra (the four vertices are distinct but aligned).
The reason is that it may not be possible to subdivide the polyhedra separately
in such a way that the subdivision of the faces is matched by the gluings,
therefore some flat tetrahedra may need to be inserted
(see \cite{PePo,PeWe} for a detailed discussion of this process).
The paper~\cite{WYY} describes a sufficient condition for the existence
of a subdivision of
the Epstein-Penner decomposition into genuine hyperbolic ideal polyhedra,
and~\cite{LST} shows that up to passing to a finite cover of $M$
one can always find a genuine hyperbolic ideal triangulation.
The experimental findings of~\cite{CaHiWe} strongly suggest
that every cusped hyperbolic 3-manifold does possess such a triangulation,
but the question is apparently open for the time being.

One further aspect is worth mentioning. SnapPea solves the hyperbolicity equations
using Newton's method and numerical approximation, so one could view
SnapPea's finding that a certain manifold is hyperbolic merely as an informal indication.
Recall however that the hyperbolicity equations are algebraic ones, and that they have at most
one solution. This implies that the solution, if any, belongs to some finite extension of
$\matQ$. Goodman's software Snap~\cite{Snap}, starting from a high-precision numerical
solution, is capable of guessing what the right extension of $\matQ$ is and then
to check that the solution is an exact one using arithmetic, without approximation.
This process has been successfully applied to the manifolds found in~\cite{CaHiWe}, which means
that this census is immune to numerical flaws.

\section{Complexity and closed manifolds}
This section represents a singularity in the present survey, since hyperbolic
geometry only plays in it a comparatively marginal r\^ole. We will briefly discuss Matveev's~\cite{mat:compl}
\emph{complexity theory} for closed manifolds, and the experimental results obtained exploiting it.
To begin, we extend the notion of special spine of a closed (orientable) $3$-manifold $M$,
already discussed in Section~\ref{objects:sec}, by defining a \emph{simple spine} of $M$
as a compact polyhedron $P\subset M$ onto which $M$ minus some number of points collapses,
and such that the link~\cite{RS} of every point of $P$ is contained in the complete
graph with $4$ vertices. Note that special polyhedra, surfaces and graphs are simple polyhedra.
For a simple polyhedron we can still define a \emph{vertex} as a point appearing as in
Fig.~\ref{special:fig}-right. Following Matveev we then define the \emph{complexity} $c(M)$ of
a manifold $M$ as the minimal number of vertices in a simple spine of $M$.

To explain the reason why the definition of complexity is based on the
more flexible notion of simple (rather than special) spine
we need to recall that a \emph{connected sum}
of two manifolds $M_1$ and $M_2$ is a manifold $M_1\#M_2$ obtained by removing
open balls from $M_1$ and $M_2$, and by gluing the resulting boundary spheres.
Since $S^2$ has only two isotopy classes of self-homeomorphisms, if $M_1$ and $M_2$ are connected
then almost two different $M_1\#M_2$'s exist. Moreover if $M_1$ and $M_2$ are \emph{oriented}
(as opposed to orientable only) and one insists that the gluing homeomorphism  should
reverse the induced orientations, one gets a uniquely defined $M_1\#M_2$.

The $3$-sphere $S^3$ is the identity element for the operation $\#$ of connected sum, and
a manifold $M$ is called \emph{prime} if it cannot be expressed as a connected sum
with both summands different from $S^3$; in addition, $M$ is called \emph{irreducible}
if every $2$-sphere in $M$ bounds a $3$-ball in $M$. Of course every
irreducible manifold is prime. The following Haken-Kneser-Milnor decomposition
theorem has been known for a long time~\cite{Hempel,Jaco}:

\begin{thm}
The only prime non-irreducible (closed orientable) $3$-manifold is $S^2\times S^1$.
Every $3$-manifold can be expressed in a unique way as a connected sum of prime ones.
\end{thm}

Turning back to complexity, let us define a simple spine $P$ of $M$ to be
\emph{minimal} if it has $c(M)$ vertices and no proper subset of $P$ is still a spine
of $M$. We now have the following fundamental result of Matveev~\cite{mat:compl}:

\begin{thm}\label{add:special:thm}
\begin{itemize}
\item[(1)] $c(M_1\#M_2)=c(M_1)+c(M_2)$;
\item[(2)] If $M$ is prime then either $c(M)=0$ and
$$M\in\{S^2\times S^1,\ S^3,\ \matP^3(\matR),\ L(3;1)\},$$
or $c(M)>0$ and every minimal simple spine of $M$ is special.
\end{itemize}
\end{thm}

Item (1) of the previous theorem, which translates into the statement that
\emph{complexity is additive under connected sum}, means that to know the complexity
of any manifold one only needs to know that of its connected summands. And item (2) implies
that, with a few easy exceptions, \emph{the complexity of a prime manifold equals the minimal
number of tetrahedra required to triangulate it}. Note that additivity would not be true
if the definition of complexity were based on special spines, or on triangulations. Employing
simple spines one has in addition the following advantage, that proves extremely useful
in practice:

\begin{prop}\label{simple:moves:prop}
A special polyhedron to which one of the moves shown in Fig.~\ref{simplemoves:fig}
    \begin{figure}
    \begin{center}
    \includegraphics[scale=.45]{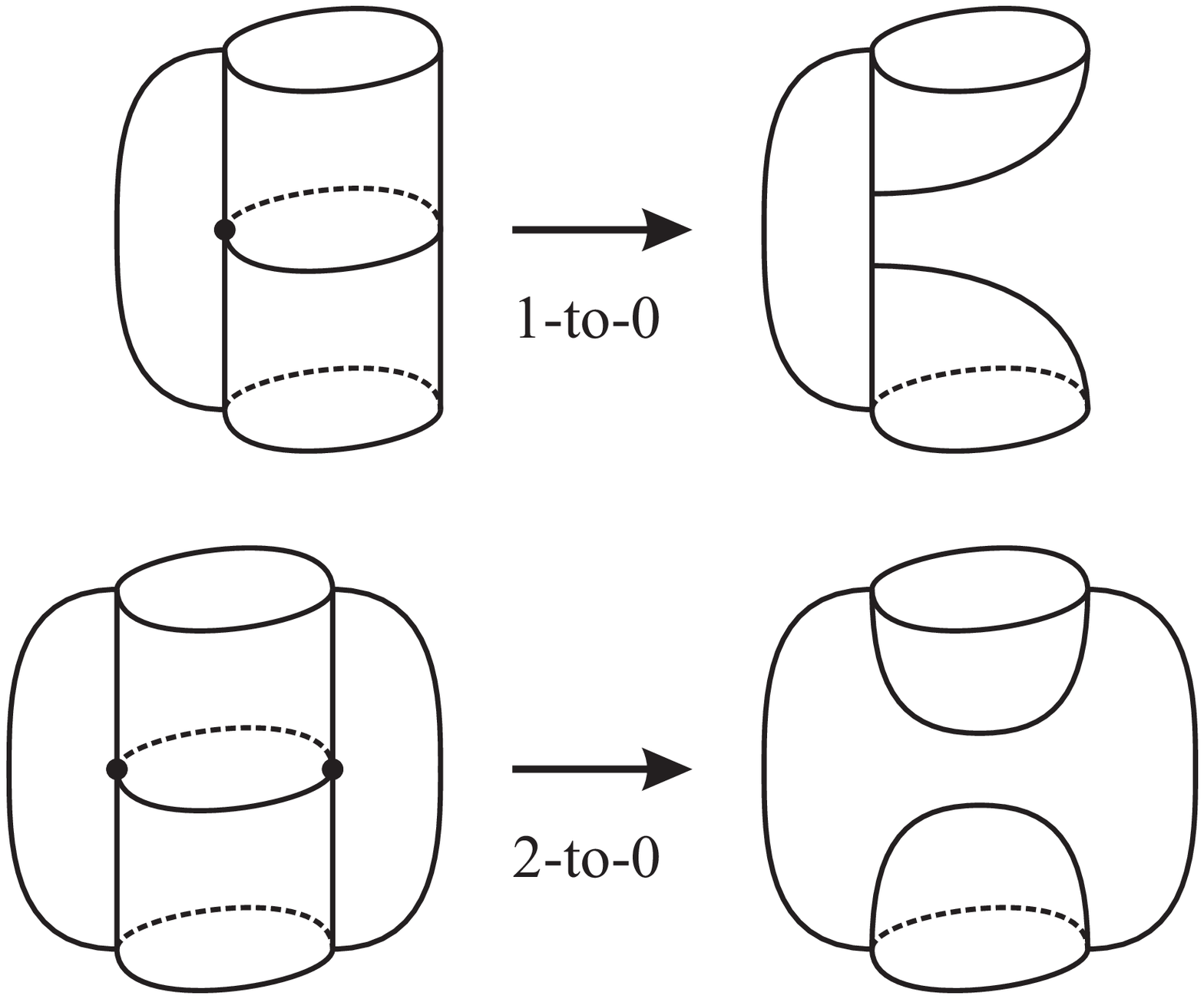}
    \mycap{Moves on special spines turning them into simple ones.}
    \label{simplemoves:fig}
    \end{center}
    \end{figure}
can be applied is not a minimal spine of a prime $3$-manifold of positive complexity.
\end{prop}

The proof of this result is given by Fig.~\ref{simplemoves:fig} itself, because the moves described in it
preserve the property that a polyhedron be a spine of a manifold, they lead to simple polyhedra, and they
reduce the number of vertices.

According to Theorem~\ref{add:special:thm}(2) and Proposition~\ref{simple:moves:prop},
given $n\geq 1$, to produce the list of closed irreducible $3$-manifolds of complexity $n$,
one can proceed according to the following (partly simultaneous) steps:
\begin{itemize}
\item Recursively construct the list of all orientable special
spines with $n$ vertices of closed manifolds (or, dually,
the list of all triangulations with $n$ tetrahedra of closed orientable manifolds);
\item During the construction, check whether the configurations of Fig.~\ref{simplemoves:fig}
appear in incompletely constructed spines and, if so, discard automatically all their possible completions;
\item Once a reduced list of spines has been obtained, eliminate duplications of manifolds by
repeated applications of the $2$-to-$3$ move and its inverse, and show that the final list contains distinct
manifolds using some invariants (such as homology or the Turaev-Viro invariants~\cite{TV}).
\end{itemize}

This strategy has been carried out by Matveev for $n\leq 8$, by Martelli and the author for $n=9$
(using a substantial refinement~\cite{expmath1} of Proposition~\ref{simple:moves:prop}, based
on a certain theory of \emph{bricks}), and then independently by Martelli and Matveev
(see~\cite{Brunosurv} and the references quoted there) for larger values of $n$.
We also mention that Matveev and Tarkaev~\cite{recognizer} have written a software, based
on special spines and on the idea of applying moves to simplify them,
that allows to \emph{recognize} any given closed manifold in a very efficient way;
the web site~\cite{recognizer} also includes very helpful electronic lists of manifolds.
In addition, non-orientable versions of the census have been obtained by
Amendola and Martelli~\cite{AM1,AM2} (using results of Martelli and the author~\cite{MP:ill}
on non-orientable bricks) and by Burton~\cite{burt1,burt2}.

\paragraph{Closed hyperbolic manifolds}
Matveev showed with a (complicated) theoretical argument that no closed manifold of
complexity smaller than $9$ can be hyperbolic. The following was proved in~\cite{expmath1}:

\begin{thm}
There are precisely $4$ closed orientable hyperbolic $3$-manifolds of complexity
$9$, and they coincide with those of smallest known volume.
\end{thm}

The $4$ manifolds referred to in the previous statement include the Weeks manifold,
now known to be the minimum-volume closed orientable hyperbolic one, thanks to
a result of Milley~\cite{milley} based on his joint work with Gabai and Meyerhoff
on \emph{Mom}'s~\cite{GMM}. Martelli~\cite{Brunosurv} found $25$ hyperbolic
manifolds in complexity $10$, and Martelli and Matveev found (the same!) few more in higher complexity.

To conclude we mention that two completely alternative approaches to
closed hyperbolic manifolds exist but will not be reviewed here. On one hand
one can obtain a wealth of them doing Dehn filling on cusped manifolds
(see Theorem~\ref{filling:thm}),
which SnapPea~\cite{SnapPea} allows to do very efficiently.
On the other hand one can try to construct a hyperbolic structure on a given closed
manifold, starting from a triangulation and using a method suggested by Casson~\cite{Casson}
(see also Manning~\cite{Manning}).

\section{Geodesic boundary and graphs}

As many of the ideas in the realm of hyperbolic geometry, those underlying the algorithmic hyperbolization
of manifolds with boundary are again due to Thurston~\cite{bible}, who first constructed such a structure
on the complement of a graph in the $3$-sphere. As in the case of cusped manifolds, where one starts from
some hyperbolic tetrahedra and imposes matching and completeness of the structure induced on
the manifold obtained by gluing them, one starts from certain parameterized
building blocks and tries to solve a system of equations. To describe the building blocks
we will need a model of hyperbolic space not employed so far, namely the \emph{projective model},
obtained by projecting radially (whence the name) the hyperboloid $\calH^n_+$ to the
unit disc $B^n_\matP$ of the hyperplane
at height $x_0=1$ in Minkowski space $\matR^{1,n}$.
The main advantage of this model is that hyperbolic straight lines appear as Euclidean straight segments
in it (but the angles are not the Euclidean ones, as it happens instead in the disc and half-space models).

We define a \emph{hyperideal hyperbolic polyhedron} as a polyhedron $P$ in the space $\matR^n$ containing
$B^n_\matP$ so that:
\begin{itemize}
\item $P$ has some \emph{hyperideal} vertices, lying outside the closure of
$B^n_\matP$, and possibly some \emph{ideal} ones, lying on the boundary of $B^n_\matP$;
\item $P$ has some genuine edges, meeting the interior of $B^n_\matP$,
and possibly some \emph{ideal} edges, tangent at one point to the boundary of $B^n_\matP$;
\item The ends of each ideal edge are hyperideal (\emph{i.e.}, not ideal) vertices of $P$.
\end{itemize}
To such a $P$ we will always associate the corresponding \emph{truncated hyperideal hyperbolic polyhedron} $\widehat{P}$,
to define which we introduce for each hyperideal vertex $v$ of $P$:
\begin{itemize}
\item The $(n-1)$-sphere $\gamma_v\subset\partial B^n_\matP$
of the tangency points to $\partial B^n_\matP$ of lines emanating from $v$;
\item The hyperplane $\alpha_v$ in $\{1\}\times\matR^n\subset\matR^{1,n}$ containing $\gamma_v$;
\item The closed half-space $Q_v$ in $\{1\}\times\matR^n$
bounded by $\alpha_v$ and not containing $v$.
\end{itemize}
Then we define the truncation $\widehat{P}$ of $P$ as the intersection of $P$ with
the $Q_v$'s as $v$ varies among the hyperideal vertices of $P$.
See Fig.~\ref{hyperid:fig} for $2$-dimensional examples of $P$ and $\widehat{P}$.
    \begin{figure}
    \begin{center}
    \includegraphics[scale=.45]{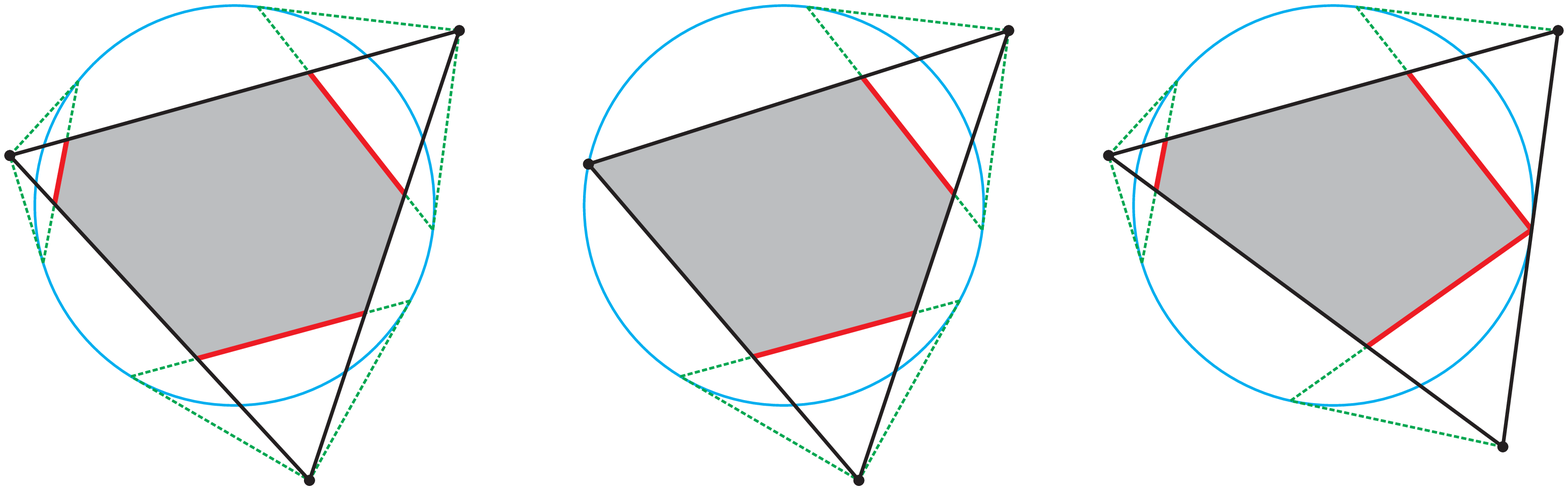}
    \mycap{A genuine hyperideal triangle and its truncation;
    a hyperideal triangle with an ideal vertex and its truncation;
    a hyperideal triangle with an ideal edge and its truncation.}
    \label{hyperid:fig}
    \end{center}
    \end{figure}
Note that one can naturally define for a truncated hyperideal hyperbolic polyhedron the
\emph{truncation faces} as those lying on the $\alpha_v$'s, and the \emph{lateral faces},
coming from the original ones. The following fact is easy to show:

\begin{lemma}
The truncation faces and the lateral faces of a
truncated hyperideal hyperbolic polyhedron lie at right angles
to each other.
\end{lemma}

This implies that the hexagon in Fig.~\ref{hyperid:fig}-left
is right-angled (even if one would not be able to tell from the picture),
and that the two pentagons in Fig.~\ref{hyperid:fig}-center and -right are right-angled
at the non-ideal vertices.

Turning to dimension $3$, the next result follows from the previous one and will be needed below:

\begin{lemma}\label{rectangle:lem}
If $v$ is the ideal point at which an ideal edge of a
truncated hyperideal hyperbolic polyhedron $\widehat{P}$ meets
$\partial B^3_\matP$, then the intersection of $\widehat{P}$
with a sufficiently small horosphere centered at $v$ is a Euclidean rectangle.
\end{lemma}

\paragraph{Moduli and hyperbolicity equations}
Let us from now on restrict our discussion to dimension $3$.
Given a compact manifold $M$ with boundary, or more generally a pair $(M,\calA)$ where
$\calA\subset\partial M$ is a family of closed annuli, one tries to construct on $M$
a hyperbolic structure such that the non-toric components of $\partial M\setminus \calA$
are totally geodesic (so the components of $\calA$ give annular cusps,
whereas the tori give toric cusps). The way to do this algorithmically is again to start
from an ideal triangulation $\calT$, encode by certain modules the structures of truncated
hyperideal hyperbolic tetrahedra one can put on the tetrahedra of $\calT$, and
try to solve certain equations on the modules that translate the fact that the structures of the tetrahedra
match to give a global complete hyperbolic structure on $M$ of the appropriate type.
As for modules, the following was shown in~\cite{fujii}
(see also~\cite{fp}):

\begin{prop}
The space of modules for a hyperideal hyperbolic tetrahedron is given by the $6$ dihedral
angles, that vary freely subject to the following restrictions:
\begin{itemize}
\item The sum of the angles at a hyperideal vertex is less than $\pi$;
\item The sum of the angles at an ideal vertex is equal to $\pi$;
\item The angle at an ideal edge is $0$.
\end{itemize}
\end{prop}

Since this will be needed soon, we also mention that,
given a choice of modules for a hyperideal hyperbolic tetrahedron $\Delta$, the lengths
of all the lateral and truncation edges of the corresponding truncated $\widehat{\Delta}$ are computed by explicit
formulae to be found in~\cite{fp}. Note that an ideal edge always has length $0$, and an edge
with one or both ideal ends has length $+\infty$.

Turning to the global matching of the structures on the individual tetrahedra of a triangulation,
we begin from the case of a pair $(M,\calA)$ with $\calA=\emptyset$, where one starts from an ideal
triangulation $\calT$ of a compact $M$ in the usual sense. (The case $\calA\neq\emptyset$ requires important
variations discussed below.) In this case the matching equations express the fact that
the lengths of all the truncation and lateral edges should be matched by the gluings of $\calT$.
But, as a matter of fact, since a right-angled
hyperbolic hexagon with finite vertices is determined
by the lengths of three pairwise non-consecutive edges, the requirement that \emph{all} lengths
should be matched is typically redundant.

As just described, the matching equations for hyperideal tetrahedra
are quite different than those for ideal tetrahedra. In particular, we stress that
the formulae to compute the lengths of the edges involve trigonometric and hyperbolic functions,
so the equations are \emph{not} algebraic ones. On the other hand
the completeness equations are precisely the same: the modules give
Euclidean structures up to similarity on the triangles of which each boundary torus
is constituted, and completeness translates into the fact that each such torus
should be Euclidean, and in turn into explicit equations in the modules
along the lines of Proposition~\ref{compl:eq:prop}.

\paragraph{Annular cusps}
For a pair $(M,\calA)$ with $\calA\subset\partial M$ a family of closed annuli, the approach
to hyperbolization using triangulations requires an important variation. To understand it,
suppose that $(M,\calA)$ has a decomposition $\calT$ into truncated hyperideal hyperbolic tetrahedra.
Then an annular cusp $A\in\calA$ corresponds to an ideal edge of $\calT$.
More precisely, one obtains the compact pair $(M,\calA)$ by first taking the compact manifold
obtained by gluing the truncated versions of the tetrahedra in $\calT$, and then
digging open cylindrical tunnels along the edges of $\calT$. This means that $\calT$ itself
is not an ideal triangulation of $M$. On the contrary, the following holds:

\begin{prop}
The ideal triangulations required to hyperbolize a pair $(M,\calA)$ are
those of the form $(\calT,a)$, where:
\begin{itemize}
\item $\calT$ is an ideal triangulation of the manifold $N$ described next, and $a$ is a family of edges of $\calT$;
\item $N$ is the manifold obtained from $M$ by gluing a solid cylinder
(a $2$-handle) to each annulus in $\calA$;
\item The family of edges $a$, viewed in $N$, is precisely the family of cores of the solid cylinders glued to $M$ to get $N$;
\item When choosing modules for the tetrahedra in $\calT$, the ideal edges should be precisely those in $a$.
\end{itemize}
\end{prop}

The case with annular cusps requires the initial subtlety just described,
and one more. The point is that
there is one very special case where two ``hexagons'' with the same ordered lengths
of the edges need not be isometric, so the matching of lengths of the edges is
not sufficient to ensure consistency of the hyperbolic structure carried
by a choice of modules for the tetrahedra in a triangulation,
and an additional equation must be added.
This occurs when a hexagon has one
ideal edge and an opposite ideal vertex,
so it reduces to a quadrilateral with two ideal and two finite vertices.
The extra parameter describing the shape of such an object is described in~\cite{fp}
together with the method to compute it starting from the modules.

Fortunately enough, after these two complications, we can show that
in dealing with annular cusps no completeness issues arise:

\begin{prop}
Consider a (possibly incomplete) hyperbolic structure on $(M,\calA)$ given by a solution of the matching
equations for a triangulation $(\calT,a)$ as in the previous result. Then the structure is
automatically complete at the annular cusps $\calA$.
\end{prop}

\begin{proof}
By Lemma~\ref{rectangle:lem} a horospherical cross-section at some $A\in\calA$
is obtained by gluing some rectangles, so it is a Euclidean annulus with boundary
components of equal length. The double of such an annulus is a Euclidean torus (and not merely
a similarity one), and the conclusion easily follows from Proposition~\ref{compl:prop}.
\end{proof}

\paragraph{Canonical decomposition}
The recognition of hyperbolic $3$-manifolds with geodesic boundary is based on an analogue
for this type of manifolds of the Epstein-Penner canonical decomposition, due to Kojima~\cite{koji1,koji2}.
For a pair $(M,\calA)$ without toric cusps (but $\calA$ can be non-empty, so annular cusps are allowed)
the Kojima decomposition is a subdivision of $M$ into truncated hyperideal hyperbolic polyhedra,
possibly with ideal edges but without ideal vertices, and it is simply dual to the cut-locus of the
boundary, as illustrated in Fig.~\ref{cutlocus:fig}. This definition is of course of impractical use,
but Kojima proved an analogue of Proposition~\ref{hull:prop} that allows the actual computation
of the canonical decomposition. To state this result we need to recall more of the geometry of
the hyperboloid model $\calH^n_+$ of hyperbolic space. We define the $1$-sheeted hyperboloid
$$\calS^n=\{y\in\matR^{1,n}:\ \scalgen xx=+1\}$$
and, for $y\in\calS^n$,
$$Q_y=\{x\in\calH^n_+:\ \scalgen xy\leq 0\},$$
noting that $Q_y$ is a geodesic half-space in $\calH^n_+$, and that each such half-space has
the form $Q_y$ for a unique $y\in\calS^n$.

Let us now consider a hyperbolic $(M,\calA)$ without toric cusps and recall that the hyperbolic structure induces
an identification between its universal cover and an intersection of closed geodesic half-spaces
in $\matH^3$. Using the hyperboloid model $\calH^3_+$ we then have $\widetilde{M}=\bigcap\{Q_y:\ y\in\calP\}$
for some family of points $\calP\subset\calS^3$. As in the cusped case we now define $C$ as the
convex hull of $\calP$ in $\matR^{1,3}$. Kojima proved the following result, stated
in a rather informal way here (but carefully stated and proved in~\cite{fp}):

\begin{prop}
If $(M,\calA)$ is hyperbolic without toric cusps then
the canonical decomposition of $(M,\calA)$ dual to the cut-locus of the boundary
is obtained by projecting radially to $\calH^3_+$ the faces of $\partial C$
that meet the positive time-like half-lines.
\end{prop}

Turning to the case of a hyperbolic manifold with geodesic boundary $(M,\calA)$ also having
toric cusps, we briefly mention that a canonical decomposition has been constructed by Kojima also
in this case. The argument is somewhat more complicated, the main steps being as follows:
\begin{itemize}
\item Let $\calP_\calS\subset\calS^3$ be the family of points such that
the universal cover of $(M,\calA)$ is $\bigcap\{Q_y:\ y\in\calP_\calS\}$;
\item Let $\calP_\calL\subset\calL^3_+$ be the family of points such that
the family of horoballs
$\{B_y:\ y\in\calP_\calL\}$ projects in $(M,\calA)$ to a family of
equal-volume ``sufficiently small'' disjoint toric cusps;
\item Let $\calP=\calP_\calS\cup\calP_\calL$ and define $C$ as the convex hull
in $\matR^{1,3}$ of $\calP$;
\item Then the canonical decomposition of $(M,\calA)$ is obtained by
projecting radially to $\calH^3_+$ the faces of $\partial C$
that meet the positive time-like half-lines,
and then suitably subdividing those arising from vertices in $\calP_\calL$.
\end{itemize}
We only mention that how small the toric cusps should be in order for this
construction to work was left implicitly determined by Kojima,
and was later spelled out in a quantitative fashion in~\cite{fp}.

\paragraph{Algorithmic recognition}
While enumerating some class of hyperbolic manifolds with geodesic boundary,
for each manifold one
constructs the structure using a triangulation (which, in practice, always works for minimal
triangulations if there are no topological obstructions to hyperbolicity), and then
one is faced with the issue of algorithmically finding the Kojima canonical decomposition.
The strategy to do so is the same as in the cusped context: starting from the given triangulation
one tries to decide whether its tetrahedra represent the projections
of the faces of $\partial C$, which amounts to checking whether the angles between
suitable liftings of the tetrahedra, determined by the global geometry, are convex if viewed from the origin.
And this can be carried out using an extension of the Sakuma-Weeks tilt formula, due to
Usijima~\cite{akira:deco} and carefully described in~\cite{fp}.
If some concave angle is found the combinatorics of the triangulation is changed by performing
the $2$-to-$3$ move along the offending face, until the process gets stuck (which never happens in practice)
or the Kojima decomposition is reached.

\paragraph{Experimental results}
The framework described above was successfully used by Frigerio, Martelli and the
author~\cite{fmp3} to list all manifolds with non-empty geodesic boundary and (possibly) toric cusps,
but no annular cusp, that can be triangulated with up to $4$ tetrahedra.
The data are available online~\cite{www:CP} and include the computation of the volume,
based on results of Ushijima~\cite{akira:vol}.

One of the most striking findings of~\cite{fmp3} is that there are $56$ manifolds whose canonical
Kojima decomposition consists of a single hyperideal regular octahedron (with different combinatorics of
the gluings). The corresponding $56$ manifolds share the same volume and would be extremely difficult
to distinguish from each other using different techniques (such as the invariants of algebraic topology
or those of Turaev and Viro): it is only using hyperbolic geometry in its full power that one can
tell that they are actually distinct. We also mention that this result naturally prompted the problem
of enumerating all the different manifolds that can be obtained gluing the faces of the
octahedron, solved by Heard, Pervova and the author in~\cite{HPP}.

Turning to graphs, the trivalent hyperbolic ones in the most general sense (with parabolic meridians)
were investigated by Heard, Hodgson, Martelli and the author in~\cite{HHMP},
where (with the restriction that each graph should have at least one trivalent vertex)
all those that can be triangulated by $5$ or less tetrahedra
were enumerated and carefully analized. The enumeration and analysis have exploited Heard's excellent
software~\cite{Orb}, which allows to hyperbolize and study manifolds with boundary and orbifolds in
an extremely effective fashion.

No systematic enumeration of hyperbolic orbifolds has been carried out so far, but the theoretical
and computer methods are all in place, as described above, and the author is hoping to
contribute to the topic in the future.

\vspace{1cm}

\noindent Dipartimento di Matematica Applicata\\
Via Filippo Buonarroti, 1C\\
56127 PISA -- Italy\\
{\tt petronio@dm.unipi.it}

\end{document}